\documentclass[12pt]{article}
{\begin{Sbox}\begin{minipage}}%
{\end{minipage}\end{Sbox}\fbox{\TheSbox}}%

\usepackage{authblk}
\usepackage[psamsfonts]{amssymb}
\usepackage{amsmath, amsthm, anysize, enumerate, color, url}
\usepackage{graphicx}
\usepackage{epstopdf}
\usepackage{epsfig}
\usepackage{subfigure}
\usepackage{breakurl}
\usepackage{algorithm}
\usepackage{algpseudocode}
\usepackage{natbib}
\usepackage{booktabs}
\usepackage{threeparttable}
\usepackage[colorlinks,citecolor=blue,urlcolor=blue]{hyperref}
\usepackage{threeparttable}
 \UseRawInputEncoding

\newtheorem{theorem}{Theorem} 

\newtheorem{lemma}{Lemma} 

\newtheorem{assumption}{Assumption}
\theoremstyle{definition}
\newtheorem{remark}{Remark}  

\newcommand{\E}{\mathbb{E}}
\newcommand{\R}{\mathbb{R}}

\renewcommand{\P}{\mathbb{P}}

\newcommand{\bs}{\boldsymbol}
\newcommand{\Z}{\mathbb{Z}}

\renewcommand{\baselinestretch}{1.2}
\usepackage{setspace}

\begin{document}
\title{A network Poisson model for weighted directed networks  with covariates}
\author{
Meng Xu\thanks{Department of Statistics, Central China Normal University, Wuhan, 430079, China.
\texttt{Emails:}mengxu@mails.ccnu.edu.cn.}
\hspace{20mm}
Qiuping Wang\thanks{Corresponding Author:Department of Statistics, Central China Normal University, Wuhan, 430079, China.
\texttt{Emails:}qp.wang@mails.ccnu.edu.cn.}
\\
Central China Normal University\\
}
\date{}

\maketitle
\begin{abstract}

The edges in networks are not only binary, either present or absent, but also take weighted values in many scenarios
(e.g., the number of emails between two users). The covariate-$p_0$ model has been proposed to model
binary directed networks with the degree heterogeneity and covariates.
However, it may cause information loss when it is applied in weighted networks.
In this paper, we propose to use the Poisson distribution to model weighted directed networks, which
admits the sparsity of networks, the degree heterogeneity and the homophily caused by covariates of nodes.
We call it the \emph{network Poisson model}.
The model contains a density parameter $\mu$, a $2n$-dimensional node parameter $\bs{\theta}$ and a fixed dimensional regression coefficient $\bs{\gamma}$ of covariates. Since the number of parameters increases with $n$, asymptotic theory is nonstandard.
When the number $n$ of nodes goes to infinity, we establish the $\ell_\infty$-errors for the maximum likelihood estimators (MLEs),
$\widehat{\bs{\theta}}$ and $\widehat{\bs{\gamma}}$, which
are $O_p( (\log n/n)^{1/2} )$ for $\widehat{\bs{\theta}}$ and $O_p( \log n/n)$ for $\widehat{\bs{\gamma}}$, up to
an additional factor.
We also obtain the asymptotic normality of the MLE.
Numerical studies and a data analysis demonstrate our theoretical findings.


\end{abstract}
\vskip 5pt

\section{Introduction}

\subsection{Background}Many complex interactive behaviors can be conveniently represented as networks or
graphs, where nodes denotes entities depending on different contexts  and edges denote interactions between entities.
Examples include friendships between people in social networks, emails between users in email networks,
hyperlinks between internet webs in hyperlink networks,
citations between papers and authors in citation networks, following behaviors between blogs in social media such as Twitter,
chemical reactions between proteins in biological networks, to just name a few.
Many statistical methodologies have been developed to analyze network data;
see \cite{Goldenberg2010}, \cite{Fienberg2012}, \cite{Robins2007173}, \cite{ROBINS2007192},
\cite{Bhattacharyya-Bickel-2016} and \cite{KimSS2018} for reviews and references therein.
The monograph by \cite{Kolaczyk2009} provides a comprehensive introduction on statistical analysis of network data.

Networks could be undirected or directed, weighted or unweighted.
Most realistic networks exhibit three typical features including sparsity, degree heterogeneity and homophily.
Sparsity means that the density of networks is small, in which many nodes do not have direct connections.
The degree heterogeneity describes such phenomenon that degrees of nodes vary greatly. Some nodes may have many connections
while others may have relatively less connections. The homophily characterizes
the tendency that individuals with same attributes such as age and sex are easier to form connections.
For example, the directed friendship network in \cite{Lazega:2001}
displays a strong homophily effect as shown in \cite{Yan-Jiang-Fienberg-Leng2018}.

One of the most popular models to model  the degree heterogeneity in undirected networks
is the  $\beta$-model [\cite{Chatterjee:Diaconis:Sly:2011}] that assigns one degree parameter to each node.
It is an undirected version of the well-known $p_1$ model in \cite{Holland:Leinhardt:1981}.
Asymptotic theory is nonstandard because of an increasing dimension of parameter space.
Exploring theoretical properties in the $\beta$-model and its variants has received wide attentions in recent years
[\cite{Chatterjee:Diaconis:Sly:2011, Perry:Wolfe:2012, Hillar:Wibisono:2013, Yan:Xu:2013, Rinaldo2013,Yan:Qin:Wang:2015,Graham2017,yan2018a,Chen:2020}].
In particular, \cite{Chatterjee:Diaconis:Sly:2011} proved the uniform consistency of the maximum likelihood estimator (MLE); \cite{Yan:Xu:2013} derived its asymptotic normality by approximating the inverse of the Fisher information matrix.
In the directed case, \cite{Yan:Leng:Zhu:2016a} proved the consistency and asymptotic normality of the MLE in the $p_0$ model
that is an exponential random graph model with out-degree and in-degree sequences as sufficient statistics. In the framework of
non-exponential random graph model, \cite{Wang:Wang:Luo2020} proposed Probit model to model the degree heterogeneity of the directed networks and proved the consistency and asymptotic normality of the moment estimator. Besides, \cite{Wang:Zhao:Luo2021} use the probit distribution to model the degree heterogeneity of the affiliation networks and established the uniform consistency and the asymptotic normality of the moment estimator.

\cite{Yan-Jiang-Fienberg-Leng2018} proposed the covariate-$p_0$ model to model the aforementioned three network features
in unweighted directed networks. They established the consistency and asymptotic normality of the restricted MLE
by using the restricted maximum likelihood method because of the challenge of exploring asymptotic theory [\cite{Fienberg2012,Graham2017}].
\cite{wang2021explicit} further incorporated a sparsity parameter to the covariate-$p_0$ model to allow the sparsity
and developed the unrestricted maximum likelihood theory including the consistency and asymptotic normality of the MLE.
However, the covariate-$p_0$ model is only designed to unweighted directed networks. When we apply it to weighted networks,
we need to neglect the weight information (i.e., treating all positive weight value as ``1" and others as ``0").
This may cause the information loss. As one example, all covariates are not significant when we apply the covariate-$p_0$ model
to fit the well-known Enran email network [\cite{Cohen2004}]; see Table \ref{Table:gamma:sparsep0}.
This motivates the present paper. We extend the covariate-$p_0$ model
to weighted networks by using the Poisson distribution to model weighted edges.
\subsection{The Model}
We now introduce our model.
Consider a weighted directed graph $G_{n}$ on $n$ ($n \geqslant 2$) nodes labeled by $1,\cdots, n$. Let $A=(a_{ij})_{n\times n}$ be the adjacency matrix $G_{n}$, where $a_{ij}\in \{0,1,\ldots\}$ is the weight of the directed edge from head node $i$ to tail node $j$. We do not consider self-loops here, i.e., $a_{ii}=0$. Let $d_{i}=\sum_{j=1,j\neq i}^{n}a_{ij}$ be the out-degree of vertex $i$ and $\bs{d}= (d_{1},\cdots,d_{n})$ be the out-degree sequence of the graph G. Similarly, define $b_{j}=\sum_{i=1,i\neq j}^{n}a_{ij}$ as the in-degree of vertex $j$ and $\bs{b}= (b_{1},\cdots,b_{n})$ as the
in-degree sequence. The pair $\{\bs{b},\bs{d}\}$ or $\{(b_{1}, d_{1}),\cdots , (b_{n}, d_{n})\}$ is the bi-degree sequence. We assume that all edges are independently distributed as Poisson random variables with the probability distributions:
\begin{equation}\label{model}
\P( a_{ij}= k) =\frac{e^{k(\mu+\alpha_{i}+\beta_{j}+Z_{ij}^\top\bs{\gamma})}}{k!}\exp(-e^{\mu+\alpha_{i}+\beta_{j}+Z_{ij}^\top\bs{\gamma}}),
1\le i \neq j \le n.
\end{equation}
The parameter $\mu$ quantifies the network sparsity.
$\alpha_{i}$ quantifies the effect of establishing outbound edges from sender $i$
while $\beta_{j}$ quantifies the effect of attracting inbound edges from receiver $j$.
The vector parameter $\bs{\gamma}$ is a $p$-dimensional
 regression coefficient for the covariate $Z_{ij}$.
 We will call the above model the \emph{network Poisson model} hereafter.

The covariate $Z_{ij}$ is either a vector associated with edges or a function of node-specific covariates.
If $X_i$ denotes  a $p$-dimensional  vector of node-level attributes, then these node-level attributes can be used to construct a vector $Z_{ij}=g(X_i, X_j)$, where
$g(\cdot, \cdot)$ is a function of its arguments. As one example, if we let $g(X_i, X_j)$ being equal to $\|X_i - X_j\|_1$, then it measures the similarity between nodes $i$ and $j$.

Motivated by techniques for the analysis of the unrestricted likelihood inference in the covariate-$p_0$ model for directed graphs in \cite{wang2021explicit},
we generalize their approaches to weighted directed graphs here.
When the number of nodes $n$ goes to infinity, we derive the $\ell_\infty$-error between the MLE
$(\widehat{\bs{\theta}}, \widehat{\bs{\gamma}})$ and its true value $(\bs{\eta}, \bs{\gamma})$.
This is done by using a two-stage Newton process
that first finds the error bound between $\widehat{\bs{\theta}}_\gamma$ and $\bs{\theta}$
with a fixed $\bs{\gamma}$ and then derives the error bound between
$\widehat{\bs{\gamma}}$ and $\bs{\gamma}$.
They are $O_p( (\log n/n)^{1/2} )$ for $\widehat{\bs{\eta}}$ and $O_p( \log n/n)$ for $\widehat{\bs{\gamma}}$, up to
an additional factor on parameters.
Further, we derive the asymptotic normality of the MLE. The asymptotic distribution of $\widehat{\bs{\gamma}}$
has a bias term while $\widehat{\bs{\theta}}$ does not have such a bias,
which collaborates the findings in \cite{Yan-Jiang-Fienberg-Leng2018}. This is because of
different convergence rates for $\widehat{\bs{\gamma}}$ and $\widehat{\bs{\theta}}$.
Wide simulations are carried out to demonstrate our theoretical findings.
In our simulations, this bias is very small, even could be neglected, which is different from
the significant bias effect in \cite{Yan-Jiang-Fienberg-Leng2018} and \cite{wang2021explicit}.
This may be due to that the weighted values attenuate the bias.
The application to the Enran email data set illustrates the utility of the proposed model.

For the remainder of the paper, we proceed as follows.
In Section \ref{section:mle}, we give the maximum likelihood estimation.
In section \ref{section:main}, we present theoretical properties of the MLE.
Numerical studies are presented in Section \ref{section:simulation}.
We provide further discussion in Section \ref{section:summary}.
The proofs of theorems are relegated to the Appendix.
All supported lemmas and detailed calculations are in the Supplementary Material.

\section{Maximum Likelihood Estimation}
\label{section:mle}

Let $\bs{\alpha}=(\alpha_{1},\ldots,\alpha_{n})^\top$ and $\bs{\beta}=(\beta_{1},\ldots,\beta_{n})^\top$.
If one transforms $(\mu,\bs{\alpha}, \bs{\beta})$ to $(\mu+2c_{1}, \bs{\alpha}-c_1+c_{2}, \bs{\beta}-c_{1}-c_{2})$,
then the probability in \eqref{model} does not change.
For the identifiability of the model, several possible restriction conditions immediately appear in our mind,
including   $\alpha_n=0$, $\beta_n=0$, or $\mu=0$, $\alpha_n=0$, or $\mu=0$, $\beta_n=0$.
When we $\alpha_n=0$, $\beta_n=0$, it will keep the density parameter $\mu$.

The logarithm of the likelihood function is
\begin{align}\label{eq:likelihood function}
\ell(\mu,\bs{\alpha},\bs{\beta
},\bs{\gamma}) = \sum_{i\neq j}(a_{ij}(\mu+\alpha_{i}+\beta_{j}+Z_{ij}^\top\bs{\gamma})-
e^{\mu+\alpha_{i}+\beta_{j}+Z_{ij}^\top\bs{\gamma}}-\log (a_{ij}!))
\end{align}
The notation $\sum_{i,j=1,i\neq j}^{n}$ is a shorthand for $\sum_{i=1}^{n}\sum_{j=1,j\neq i}^{n}$. The score equations for the vector parameters $\mu, \bs{\alpha}, \bs{\beta}$, $\bs{\gamma}$ are easily seen as
\begin{equation}\label{eq:likelihood}
\begin{array}{rlll}
\sum\limits_{i,j=1,i\neq j}^{n} a_{ij} &= &\sum\limits_{i\neq j} e^{\mu+\alpha_{i}+\beta_{j}+Z_{ij}^\top\bs{\gamma}},& \\
\sum\limits_{i,j=1,i\neq j}^{n} a_{ij}Z_{ij} &= &\sum\limits_{i\neq j} Z_{ij} e^{\mu+\alpha_{i}+\beta_{j}+Z_{ij}^\top\bs{\gamma}},& \\
d_i & =& \sum\limits_{j=1, j\neq i}^n e^{\mu+\alpha_{i}+\beta_{j}+Z_{ij}^\top\bs{\gamma}},&i=1,\ldots, n, \\
b_j & = &\sum\limits_{i=1,i\neq j}^n e^{\mu+\alpha_{i}+\beta_{j}+Z_{ij}^\top\bs{\gamma}},&j=1,\ldots, n.
\end{array}
\end{equation}
Under the restriction $\mu=0$ and $\beta_n=0$, the first equation and the last equation with $j=n$ in the above
system of equations will be excluded. Although there are a total of $2n+1+p$ equations,
the number of minimal equations is only $2n-1 + p$.

The MLE $(\widehat{\mu}, \widehat{\bs{\alpha}},
\widehat{\bs{\beta}}, \widehat{\bs{\gamma}})$ of the parameter vector
$(\mu, \bs{\alpha}, \bs{\beta}, \bs{\gamma})$ is the solution of the above equations if it exist.
The Newton-Raphson algorithm can be used to solve the above equations.
We can also simply use the function ``glm" in R language to calculate the solution.

\section{Theoretical Properties}
\label{section:main}

Let $\widetilde{\bs{\alpha}}=\mu/2+\bs{\alpha}$ and $\widetilde{\bs{\beta}}=\mu/2+\bs{\beta}$, then $\widetilde{\bs{\alpha}}+\widetilde{\bs{\beta}}=\mu+\bs{\alpha}+\bs{\beta}$.
With this reparameterized technique, we could set $\mu=0$ for convenience.
Further, we set $\beta_n=0$ jointly as the identification condition in this section.
The asymptotic properties of $\widehat{\mu}$ in the restriction $\alpha_n=\beta_n=0$ is the same as those of
$\widehat{\alpha}_n$ in the restriction $\mu=\beta_n=0$.

\emph{Notations}. Let $\R = (-\infty, \infty)$ be the real domain. For a subset $C\subset \R^n$, let $C^0$ and $\overline{C}$ denote the interior and closure of $C$, respectively.
For a vector $\mathbf{x}=(x_1, \ldots, x_n)^\top\in \R^n$, denote by $\|\mathbf{x}\|$ for a general norm on vectors with the special cases
$\|\mathbf{x}\|_\infty = \max_{1\le i\le n} |x_i|$ and $\|\mathbf{x}\|_1=\sum_i |x_i|$ for the $\ell_\infty$- and $\ell_1$-norm of $\mathbf{x}$ respectively.
When $n$ is fixed, all norms on vectors are equivalent. Let
$B(\mathbf{x}, \epsilon)=\{\mathbf{y}: \| \mathbf{x} - \mathbf{y} \|_\infty \le \epsilon\}$ be an $\epsilon$-neighborhood of $\mathbf{x}$.
For an $n\times n$ matrix $J=(J_{i,j})$, let $\|J\|_\infty$ denote the matrix norm induced by the $\ell_\infty$-norm on vectors in $\R^n$, i.e.,
\[
\|J\|_\infty = \max_{x\neq 0} \frac{ \|J\mathbf{x}\|_\infty }{\|\mathbf{x}\|_\infty}
=\max_{1\le i\le n}\sum_{j=1}^n |J_{i,j}|,
\]
and $\|J\|$ be a general matrix norm.
Define the matrix maximum norm: $\|J\|_{\max}=\max_{i,j}|J_{ij}|$.
The notation $\sum_i$ denotes the summarization over all $i=1, \ldots, n$ and
$\sum_{i\neq j}$  is a shorthand for $\sum_{i=1}^n \sum_{j=1, j\neq i}^n$.
The notation $f(x) \lesssim g(x)$ or $f(x)=O(g(x))$ means that there exists a constant $c$ such that
$|f(x)| \le c |g(x)|$.

For convenience of our theoretical analysis, define $\bs{\theta}=(\alpha_1, \ldots, \alpha_n, \beta_1, \ldots, \beta_{n-1})^\top$.
We use the superscript ``*" to denote the true parameter under which the data are generated.
When there is no ambiguity, we omit the super script ``*".
Further, define
\[
\pi_{ij}:=\alpha_i+\beta_j+Z_{ij}^\top \bs{\gamma}, ~~ \lambda( x ): = e^x.
\]
Write $\lambda^\prime$, $\lambda^{\prime\prime}$ and $\lambda^{\prime\prime\prime}$ as the first, second and third derivative of $\lambda(x)$ on $x$, respectively.
Direct calculations give that $\lambda^\prime(x)=\lambda^{\prime\prime}(x)=\lambda^{\prime\prime\prime}(x)=e^x$.
Let $\epsilon_{n1}$ and $\epsilon_{n2}$ be two small positive number. When $\bs{\theta}\in B(\bs{\theta}^{\ast},\epsilon_{n1})$, $\bs{\gamma}\in B(\bs{\gamma}^{\ast},\epsilon_{n2})$, we have
\begin{align*}
|\alpha_i+\beta_j+Z_{ij}^{T}\bs{\gamma}|\leqslant \max_{i\neq j}\left|\alpha_{i}^{\ast}+\beta_{j}^{\ast}\right|+2\epsilon_{n1}+pq\|\bs{\gamma}^{*}\|_{\infty}+pq\epsilon_{n2} :=\rho_n
\end{align*}
where $q :=\max_{i,j}\|Z_{ij}\|_{\infty}$, and we assume that $q$ is a constant.
It is not difficult to verify, When $\bs{\theta}\in B(\bs{\theta}^{\ast},\epsilon_{n1})$ and $\bs{\gamma}\in B(\bs{\gamma}^{\ast},\epsilon_{n2})$,
\begin{equation}\label{ineq-mu-third-derivative}
e^{-\rho_{n}}\leqslant \lambda(\pi_{ij}) =\lambda^\prime(\pi_{ij})=\lambda^{\prime\prime}(\pi_{ij})
=\lambda^{\prime\prime\prime}(\pi_{ij}) \leqslant e^{\rho_{n}}.
\end{equation}
When causing no confusion, we will simply write $\lambda_{ij}$ stead of $\lambda_{ij}(\bs{\theta}, \bs{\gamma})$ for shorthand,
where
\[
\lambda_{ij}(\bs{\theta}, \bs{\gamma}) =e^{\alpha_i + \beta_j + Z_{ij}^\top \bs{\gamma} }=\lambda(\pi_{ij}).
\]
Write the partial derivative of a function vector $F(\widehat{\bs{\theta}}, \bs{\gamma})$ on $\bs{\theta}$ as
\[
\frac{ \partial F(\widehat{\bs{\theta}}, \widehat{\bs{\gamma}}) }{ \partial \bs{\theta}^\top }
=\frac{ \partial F(\bs{\theta}, \bs{\gamma}) }{ \partial \bs{\theta}^\top }\bigg |_{\bs{\theta}=\widehat{\bs{\theta}},
\bs{\gamma}=\widehat{\bs{\gamma}}}.
\]

Throughout the remainder of this paper, we make the following assumption.
\begin{assumption}
Assume that $p$, the dimension of $Z_{ij}$, is fixed and that the support of $Z_{ij}$ is $\Z^p$, where $\Z$ is a compact subset of $\R$.
\end{assumption}

The above assumption is made in \cite{Graham2017} and \cite{Yan-Jiang-Fienberg-Leng2018}.
In many real applications, the attributes of nodes have a fixed dimension
and $Z_{ij}$ is bounded.
As one example, if nodal variables are indicator such as sex, then the assumption holds.
If $Z_{ij}$ is not bounded, we could make the transform $\tilde{Z}_{ij} = e^{Z_{ij}}/( 1 + e^{Z_{ij}})$.

\subsection{Consistency}\label{subsection:con}

In order to establish asymptotic properties of $(\widehat{\bs{\theta}}, \widehat{\bs{\gamma}})$,
 we define a system of functions
\begin{align*}
F_{i}( \bs{\theta},\bs{\gamma})&=\sum_{j=1,j\neq i}^{n} \lambda_{ij}(\bs{\theta},\bs{\gamma})-d_{i},~i=1,\ldots,n,\\
F_{n+j}( \bs{\theta},\bs{\gamma})&=\sum_{i=1,i\neq j}^{n}\lambda_{ij}(\bs{\theta},\bs{\gamma})-b_{j},~j=1,\ldots,n,\\
F( \bs{\theta},\bs{\gamma})&=(F_{1}(\bs{\theta},\bs{\gamma}),\ldots ,F_{n}(\bs{\theta},\bs{\gamma}),F_{n+1}(\bs{\theta},\bs{\gamma}),\ldots,F_{2n-1}(\bs{\theta},\bs{\gamma}))^\top,
\end{align*}
which are based on the score equations for $\widehat{\bs{\theta}}$.
Define $F_{\gamma, i}(\bs{\theta})$ be the value of $F_{i}(\bs{\theta}, \bs{\gamma})$ when $\bs{\gamma}$ is fixed.
Let $\widehat{\bs{\theta}}_\gamma$ be a solution to $F_\gamma( \bs{\theta})=0$ if it exists.
Correspondingly, we define two functions for exploring the asymptotic behaviors of the estimator of $\bs{\gamma}$:
\begin{eqnarray}
\label{definition-Q}
Q(\bs{\theta}, \bs{\gamma})= \sum_{i,j=1;i\neq j}^n Z_{ij} (   \lambda_{ij}(\bs{\theta}, \bs{\gamma}) - a_{ij}  ), \\
\label{definition-Qc}
Q_c(\bs{\gamma})= \sum_{i,j=1;i\neq j}^n Z_{ij} (    \lambda_{ij}(\widehat{\bs{\theta}}_\gamma, \bs{\gamma}) - a_{ij}  ).
\end{eqnarray}
$Q_c(\bs{\gamma})$ could be viewed as the concentrated or profile function of $Q(\bs{\theta}, \bs{\gamma})$ in which the degree parameter $\bs{\theta}$ is profiled out.
It is clear that
\begin{equation*}\label{equation:FQ}
F(\widehat{\bs{\theta}}, \widehat{\bs{\gamma}})=0,~~F_\gamma( \widehat{\bs{\theta}}_\gamma)=0,~~
Q(\widehat{\bs{\theta}}, \widehat{\bs{\gamma}})=0,~~Q_c( \widehat{\bs{\gamma}})=0.
\end{equation*}

We first present the error bound between $\widehat{\bs{\theta}}_{\gamma}$ with $\bs{\gamma} \in B(\bs{\gamma}^*, \epsilon_{n2})$ and $\bs{\theta}^*$. This is proved by
constructing the Newton iterative sequence $\{\bs{\theta}^{(k+1)}\}_{k=0}^\infty$ with  initial value $\bs{\theta}^*$
and obtaining a geometrically fast convergence rate of the iterative sequence,
where $\bs{\theta}^{(k+1)} = \bs{\theta}^{(k)} - [ F_\gamma'(\bs{\theta}^{(k)})]^{-1} F_\gamma(\bs{\theta}^{(k)})$ and
$F_\gamma'(\bs{\theta})=\partial F_\gamma(\bs{\theta})/\partial \bs{\gamma}^\top$.
Details are given in the Supplementary Material.
The error bound is stated below.

\begin{lemma}\label{lemma:alpha:beta-fixed}
If $\bs{\gamma} \in B( \bs{\gamma}^*, \epsilon_{n2} )$ and $e^{\rho_{n}}= o( (n/\log n)^{1/26})$,
then as $n$ goes to infinity,
with probability at least $1-O(1/n)$, $\widehat{\bs{\theta}}_\gamma$ exists
and satisfies
\[
 \|\widehat{\bs{\theta}}_\gamma - \bs{\theta}^* \|_\infty
 = O_p\left(  e^{7\rho_{n}} \sqrt{\frac{ \log n}{n} } \right)=o_p(1).
\]
Further, if $\widehat{\bs{\theta}}_\gamma$ exists, then it is unique.
\end{lemma}

By the compound function derivation law, we have
\begin{eqnarray}\label{equ-derivation-a}
0=\frac{ \partial F_\gamma(\widehat{\bs{\theta}}_\gamma) }{\partial \bs{\gamma}^\top}
 = \frac{ \partial F(\widehat{\bs{\theta}}_\gamma, \gamma) }{\partial \bs{\theta}^\top}
  \frac{\partial \widehat{\bs{\theta}}_\gamma }{ \partial \bs{\gamma}^\top} + \frac{\partial F(\widehat{\bs{\theta}}_\gamma,  \bs{\gamma})}{\partial \bs{ \gamma}^\top},
  \\
  \label{equ-derivation-b}
\frac{ \partial Q_c({ \bs{\gamma}})}{ \partial { \bs{\gamma}}^\top} = \frac{\partial Q(\widehat{\bs{\theta}}_\gamma, { \bs{\gamma}})}{\partial \bs{\theta}^\top}
 \frac{\partial \widehat{\bs{\theta}}_\gamma }{{\partial \bs{\gamma}}^\top} + \frac{ \partial Q(\widehat{\bs{\theta}}_\gamma, { \bs{\gamma}}) }{ \partial { \bs{\gamma}}^\top}.
\end{eqnarray}
By solving $\partial\widehat{\bs{\theta}}_\gamma / \partial { \bs{\gamma}}^\top$ in \eqref{equ-derivation-a}
and substituting it into \eqref{equ-derivation-b}, we get
the Jacobian matrix
$Q_c^\prime( \bs{\gamma} )$ $(=\partial Q_c(\bs{\gamma})/\partial \bs{\gamma}^\top)$:
\begin{eqnarray}\label{equation:Qc-derivative}
\frac{ \partial Q_c(\bs{\gamma}) }{ \partial \bs{\gamma}^\top }  =
\frac{ \partial Q(\widehat{ \bs{\theta} }_\gamma, \bs{\gamma}) }{ \partial \bs{\gamma}^\top}
 - \frac{ \partial Q(\widehat{ \bs{\theta} }_\gamma, \bs{\gamma}) }{\partial \bs{\theta}^\top}
 \left[\frac{\partial F(\widehat{ \bs{\theta} }_\gamma,\bs{\gamma})}{\partial \bs{\theta}^\top}  \right]^{-1}
\frac{\partial F(\widehat{\bs{\theta}}_\gamma,\bs{\gamma})}{\partial \bs{\gamma}^\top}.
\end{eqnarray}
The asymptotic behavior of $\widehat{\bs{\gamma}}$ crucially depends on the Jacobian matrix
$Q_c^\prime(\bs{\gamma})$.
Since $\widehat{\bs{\theta}}_\gamma$ does not have a closed form, conditions that are directly imposed on $Q_c^\prime(\bs{\gamma})$ are not easily checked.
To derive feasible conditions, we define
\begin{equation}
\label{definition-H}
H(\bs{\theta}, \bs{\gamma}) = \frac{ \partial Q(\bs{\theta}, \bs{\gamma}) }{ \partial \bs{\gamma}^\top} - \frac{ \partial Q(\bs{\theta}, \bs{\gamma}) }{\partial \bs{\theta}^\top} \left[ \frac{\partial F(\bs{\theta}, \bs{\gamma})}{\partial \bs{\theta}^\top} \right]^{-1}
\frac{\partial F(\bs{\theta}, \bs{\gamma})}{\partial \bs{\gamma}^\top},
\end{equation}
which is a general form of $ \partial Q_c(\bs{\gamma}) / \partial \bs{\gamma}$.
Because $H(\bs{\theta}, \bs{\gamma})$ is the Fisher information matrix of the profiled log-likelihood $\ell_c(\bs{\gamma})$,
we assume that it is positively definite. Otherwise, the network Poisson model will be ill-conditioned.
When $\bs{\theta}\in B(\bs{\theta}^*, \epsilon_{n1})$, we have the equation:
\begin{equation}\label{equation-H-appro}
\frac{1}{n^2} H(\bs{\theta}, \bs{\gamma}^*) = \frac{1}{n^2} H(\bs{\theta}^*, \bs{\gamma}^*) + o(1),
\end{equation}
whose proof is omitted.
Note that the dimension of $H(\bs{\theta}, \bs{\gamma})$ is fixed and every its entry is a sum of $(n-1)n$  terms.
We assume that there exists a number $\kappa_n$ such that
\begin{equation}\label{condition-Qc-gamma}
\sup_{\bs{\theta}\in B(\bs{\theta}^*, \epsilon_{n1})} \| H^{-1}(\bs{\theta}, \bs{\gamma}^*)\|_\infty \le  \frac{ \kappa_n }{ n^2}.
\end{equation}
If $n^{-2}H(\bs{\theta}, \bs{\gamma}^*)$ converges to a constant matrix, then $\kappa_n$ is bounded.
Because $H(\bs{\theta}, \bs{\gamma}^*)$ is positively definite,
\[
\kappa_n = \sqrt{p}\times \sup_{\bs{\theta}\in B(\bs{\theta}^*, \epsilon_{n1})} 1/\lambda_{\min}(\bs{\theta}),
\]
where  $\lambda_{\min}(\bs{\theta})$ is the smallest eigenvalue of $n^{-2}H(\bs{\theta}, \bs{\gamma}^*)$.
Now we formally state the consistency result.

\begin{theorem}\label{Theorem:con}
If $\kappa_n^2 e^{40\rho_n}=o(n/\log n)$, then
 the MLE $(\widehat{\bs{\theta}}, \widehat{\bs{\gamma}})$ exists with probability at least $1-O(1/n)$, and is consistent in the sense that
\begin{align*}\label{Newton-convergence-rate}
\| \widehat{\bs{\gamma}} - \bs{\gamma}^{*} \|_\infty &=  O_p\left(
 \frac{\kappa_n e^{21\rho_n}\log n }{ n }\right  )=o_p(1) \\
\| \widehat{ \bs{\theta} } - \bs{\theta}^* \|_\infty &= O_p\left(  e^{7\rho_{n}} \sqrt{\frac{ \log n}{n} } \right)=o_p(1).
\end{align*}
Further, if $\widehat{\theta}$ exists, it is unique.
\end{theorem}

From the above theorem, we can see that $\widehat{\bs{\gamma}}$ has a convergence rate $\log n/n$ and
$\widehat{\bs{\theta}}$ has a convergence rate $(\log n/n)^{1/2}$, up to an additional factor.
If $\|\bs{\gamma}^{*}\|_\infty$ and $\|\bs{\theta}^*\|_\infty$ are constants, then $\rho_n$ and $\kappa_n$ are constants such that
the condition in Theorem \ref{Theorem:con} holds.

\subsection{Asymptotic normality of $\widehat{\bs{\theta}}$ and  $\widehat{\bs{\gamma}}$}

The asymptotic distribution of $\widehat{\bs{\theta}}$ depends  crucially on
the inverse of the Jacobian matrix $F^\prime_\gamma(\bs{\theta})$, which generally
 does not have a closed form.
In order to characterize this matrix,
we introduce a general class of matrices that encompass the Fisher matrix.
Given two positive numbers $m$ and $M$ with $M \ge m >0$,
we say the $(2n-1)\times (2n-1)$ matrix $V=(v_{i,j})$ belongs to the class $\mathcal{L}_{n}(m, M)$ if the following holds:
\begin{equation}\label{eq1}
\begin{array}{l}
0\le v_{i,i}-\sum_{j=n+1}^{2n-1} v_{i,j} \le M, ~~ i=1, \ldots, 2n-1, \\
v_{i,j}=0, ~~ i,j=1,\ldots,n,~ i\neq j, \\
v_{i,j}=0, ~~ i,j=n+1, \ldots, 2n-1,~ i\neq j,\\
m\le v_{i,j}=v_{j,i} \le M, ~~ i=1,\ldots, n,~ j=n+1,\ldots, 2n,~ j\neq n+i, \\
\end{array}
\end{equation}
Clearly, if $V\in \mathcal{L}_{n}(m, M)$, then $V$ is a $(2n-1) \times (2n-1)$ diagonally dominant, symmetric nonnegative
matrix. It must be positively definite.
The definition of $\mathcal{L}_{n}(m, M)$ here is wider than that in   \cite{Yan:Leng:Zhu:2016a}, where
the diagonal elements are equal to the row sum excluding themselves for some rows.
One can easily show that $F^\prime_\gamma(\bs{\theta})$ belongs to this matrix class.
With some abuse of notation, we use $V$ to denote the Fisher information matrix for the vector
parameter $\bs{\theta}$.

To describe the exact form of elements of $V$, $v_{ij}$ for $i,j=1,\ldots, 2n-1$, $i\neq j$, we define
\begin{equation*}
u_{ij} = \lambda(\pi_{ij}), ~~ u_{ii}=0,~~
u_{i\cdot} = \sum_{j=1}^n u_{ij}, ~~ u_{\cdot j} = \sum_{i=1}^n u_{ij}.
\end{equation*}
Note that $u_{ij}$ is the variance of $a_{ij}$.
Then the elements of $V$ are
\[
v_{ij} = \begin{cases} u_{i\cdot}, & i=j=1, \ldots, n, \\
u_{i,j-n}, &i=1, \ldots, n, j=n+1, \ldots, 2n-1, j\neq i+n, \\
u_{j,i-n}, &i=n+1, \ldots, 2n, j=1, \ldots, n-1, j\neq i-n,\\
u_{\cdot j-n}, & i=j=n+1, \ldots, 2n-1, \\
0, & \mbox{others}.
\end{cases}
\]
\cite{Yan:Leng:Zhu:2016a} proposed to approximate the inverse $V^{-1}$ by the matrix $S=(s_{i,j})$, which is defined as
\begin{equation}
\label{definition:S}
s_{i,j}=\left\{\begin{array}{ll}\frac{\delta_{i,j}}{u_{i\cdot}} + \frac{1}{u_{\cdot n}}, & i,j=1,\ldots,n, \\
-\frac{1}{u_{\cdot n}}, & i=1,\ldots, n,~~ j=n+1,\ldots,2n-1, \\
-\frac{1}{u_{\cdot n}}, & i=n+1,\ldots,2n,~~ j=1,\ldots,n-1, \\
\frac{\delta_{i,j}}{u_{\cdot (j-n)}}+\frac{1}{u_{\cdot n}}, & i,j=n+1,\ldots, 2n-1,
\end{array}
\right.
\end{equation}
where $\delta_{i,j}=1$ when $i=j$ and $\delta_{i,j}=0$ when $i\neq j$.

We derive the asymptotic normality of $\widehat{\bs{\bs{\theta}}}$ by representing
$\bs{\widehat{\bs{\theta}}}$ as a function of $\mathbf{d}$ and $\mathbf{b}$ with an explicit expression and a remainder term.
This is done via applying a second Taylor expansion to $F(\widehat{\bs{\theta}}, \widehat{\bs{\gamma}})$ and
showing various remainder terms asymptotically neglect.

\begin{theorem}\label{Theorem:binary:central}
 If $\kappa_n e^{23\rho_n} = o( n^{1/2}/\log n)$, then for any fixed $k\ge 1$, as $n \to\infty$, the vector consisting of the first $k$ elements of $(\widehat{\bs{\bs{\theta}}}-\bs{\bs{\theta}}^*)$ is asymptotically multivariate normal with mean $\mathbf{0}$ and covariance matrix given by the upper left $k \times k$ block of $S$ defined in \eqref{definition:S}.
\end{theorem}

\begin{remark}
By Theorem \ref{Theorem:binary:central}, for any fixed $i$, as $n\rightarrow \infty$, the asymptotic variance of $\hat{\theta}_i$ is $1/v_{i,i}^{1/2}$,
whose magnitude is between $O(n^{-1/2}e^{\rho_n})$ and $O(n^{-1/2})$.
\end{remark}

\begin{remark}
Under the restriction $\alpha_n=\beta_n=0$, the central limit theorem for the MLE $\widehat{\mu}$ is stated as:
$(u_{n\cdot}+u_{\cdot n})^{-1/2}( \widehat{\mu} - \mu )$,
converges in distribution to the standard normality.
\end{remark}

Now, we present the asymptotic normality of $\widehat{\bs{\gamma}}$.
Let
\[
V = \frac{\partial F(\bs{\theta}^*, \bs{\gamma}^*)}{ \partial \bs{\theta}^\top}, ~~
V_{\gamma\gamma} = \frac{\partial Q(\bs{\theta}^*, \bs{\gamma}^* ) }{ \partial \bs{\gamma}^\top },~~
V_{\theta\gamma} = \frac{\partial F(\bs{\theta}^*, \bs{\gamma}^* ) }{ \partial \bs{\gamma}^\top }.
\]
Following \cite{Amemiya:1985} (p. 126), the Hessian matrix of the concentrated log-likelihood function $\ell^c (\bs{\gamma}^*)$
is $V_{\gamma\gamma} - V_{\theta\gamma}^\top V^{-1} V_{\theta\gamma}$.
To state the form of the limit distribution of $\hat{\bs{\gamma}}$, define
\begin{equation}\label{eq:I0:beta}
I_n(\bs{\gamma}^*) =  \frac{1}{(n-1)n}(V_{\gamma\gamma} - V_{\theta\gamma}^\top V^{-1} V_{\theta\gamma}).
\end{equation}
Assume that the limit of  $I_n(\bs{\gamma}^*)$ exists as $n$ goes to infinity, denoted by $I_*(\bs{\gamma}^*)$.
Then, we have the following theorem.

\begin{theorem}\label{theorem:covariate:asym}
If $e^{9\rho_n} =o(n^{1/2}/(\log n)^2)$, then as $n$ goes to infinity, the $p$-dimensional vector
$N^{1/2}(\widehat{\bs{\gamma}}-\bs{\gamma}^* )$ is asymptotically multivariate normal distribution
with mean $I_*^{-1} (\bs{\gamma}^*) B_*$ and covariance matrix $I_*^{-1}(\bs{\gamma})$,
where $N=n(n-1)$ and $B_*$ is the bias term:
\begin{equation*}\label{definition:Bstar}
B_*=\lim_{n\to\infty} \frac{1}{2\sqrt{N}} \left[\sum_{i=1}^n \frac{  \sum_{j=1,j\neq i}^n \lambda^{\prime\prime}(\pi_{ij}^*) Z_{ij} }
{  \sum_{j=1,j\neq i}^n \lambda^\prime(\pi_{ij}^*) }+\sum_{j=1}^n \frac{  \sum_{i=1,i\neq j}^n \lambda^{\prime\prime}(\pi_{ij}^*)Z_{ij} }
{ \sum_{i=1,i\neq j}^n \lambda^\prime(\pi_{ij}^*) } \right].
\end{equation*}
\end{theorem}

\begin{remark}
The limiting distribution of $\bs{\widehat{\gamma}}$ is involved with a bias term
\[
\mu_*=\frac{ I_*^{-1}(\bs{\gamma}) B_* }{ \sqrt{n(n-1)}}.
\]
Since the MLE $\bs{\widehat{\gamma}}$ is not centered at the true parameter value, the confidence intervals
and the p-values of hypothesis testing for $\bs{\widehat{\gamma}}$ may not achieve the nominal level without bias-correction under the null: $\bs{\gamma}^*=0$.
This is referred to as the so-called incidental parameter problem in econometric literature [\cite{Neyman:Scott:1948}].
The produced bias is due to different convergence rates of $\widehat{\bs{\gamma}}$ and $\widehat{\bs{\theta}}$.
As discussed in \cite{Yan-Jiang-Fienberg-Leng2018}, we could use the analytical bias correction formula:
$\bs{\widehat{\gamma}}_{bc} = \bs{\widehat{\gamma}}- \hat{I}^{-1} \hat{B}/\sqrt{n(n-1)}$,
where $\hat{I}$ and $\hat{B}$ are the estimates of $I_*$ and $B_*$ by replacing
$\bs{\gamma}$ and $\bs{\theta}$ in their expressions with their MLEs $\bs{\widehat{\gamma}}$ and
$\bs{\widehat{\theta}}$, respectively.
In the simulation in next section, we can see that
there is a little difference between uncorrected estimates and bias-corrected estimates, which is different from
the covariate-$p_0$ model for binary directed networks in \cite{Yan-Jiang-Fienberg-Leng2018}.
This may be due to that the infinitely weighted values for edges attenuate the bias effect.
\end{remark}

\section{Numerical Studies}
\label{section:simulation}
In this section, we evaluate the asymptotic results of the MLEs for model \eqref{model} through simulation
studies and a real data example.

\subsection{Simulation studies}

Similar to \cite{Yan-Jiang-Fienberg-Leng2018}, the parameter values take a linear form. Specifically,
we set $\alpha_{i}^* = ic\log n/n$ for $i=0, \ldots, n$ and let $\beta_i^*=\alpha_i^*$, $i=0, \ldots, n$ for simplicity.
Note that there are $n+1$ nodes in the simulations.
We considered four different values for $c$ as $c\in \{0, 0.2, 0.4, 0.6\}$.
By allowing the true values of $\bs{\alpha}$ and $\bs{\beta}$ to grow with $n$, we intended to assess the asymptotic properties under different asymptotic regimes.
Being different from \cite{Yan-Jiang-Fienberg-Leng2018} with $p=2$, we set $p=3$ here.
The first element of $Z_{ij}$ was generated from standard normal distribution; the second element $Z_{ij2}$
formed by letting $Z_{ij2}=|X_{i1} - X_{j1}|$, where $X_{i1}$ is the first entry
of the $2$-dimensional node-specific covariate $X_i$ is independently generated from a $Beta(2,2)$ distribution;
and the third element $Z_{ij3}=X_{i2}*X_{j2}$, where $X_{i2}$
follows a discrete distribution taking values $1$ and $-1$ with probabilities $0.3$ and $0.7$.
The density parameter is $\mu= -\log n/4$.

Note that by Theorems \ref{Theorem:binary:central}, $\hat{\xi}_{i,j} = [\hat{\alpha}_i-\hat{\alpha}_j-(\alpha_i^*-\alpha_j^*)]/(1/\hat{v}_{i,i}+1/\hat{v}_{j,j})^{1/2}$, $\hat{\zeta}_{i,j} = (\hat{\alpha}_i+\hat{\beta}_j-\alpha_i^*-\beta_j^*)/(1/\hat{v}_{i,i}+1/\hat{v}_{n+j,n+j})^{1/2}$, and $\hat{\eta}_{i,j} = [\hat{\beta}_i-\hat{\beta}_j-(\beta_i^*-\beta_j^*)]/(1/\hat{v}_{n+i,n+i}+1/\hat{v}_{n+j,n+j})^{1/2}$
are all asymptotically distributed as standard normal random variables, where $\hat{v}_{i,i}$ is the estimate of $v_{i,i}$
by replacing $(\bs{\gamma}^*, \bs{\theta}^*)$ with $(\bs{\widehat{\gamma}}, \bs{\widehat{\theta}})$.
We record the coverage probability of the 95\% confidence interval, the length of the confidence interval, and the frequency that the MLE does not exist.  The results for $\hat{\xi}_{i,j}$, $\hat{\zeta}_{i,j}$ and $\hat{\eta}_{i,j}$ are similar, thus only the results of $\hat{\xi}_{i,j}$ are reported. We simulated networks with $n=100$ or $n=200$.
Each simulation is repeated $5,000$ times.

Table \ref{Table:alpha} reports the $95\%$ coverage frequencies for $\alpha_i - \alpha_j$ and
the length of the confidence interval.
As we can see, the length of the confidence interval increases with $c$ and decreases with $n$,
which qualitatively agrees with the theory.  All simulated coverage frequencies are all close to the nominal level.
This indicates that the conditions in the theorems may be relaxed greatly.

{\renewcommand{\arraystretch}{1}
\begin{table}[!h]\centering
\caption{The reported values are the coverage frequency ($\times 100\%$) for $\alpha_i-\alpha_j$ for a pair $(i,j)$ / the length of the confidence interval. The pair $(0,0)$ indicates
those values for the density parameter $\mu$.}
\label{Table:alpha}
 \begin{threeparttable}
\begin{tabular}{ccccccc}
\hline
n       &  $(i,j)$ & $c=0$ & $c=0.2$ & $c=0.4$ & $c=0.6$ \\
\hline
100         &$(1,2)   $&$ 94.36 / 0.529 $&$ 94.72 / 0.408 $&$ 94.54 / 0.305 $&$ 95.1 / 0.222 $ \\
            &$(50,51) $&$ 94.42 / 0.529 $&$ 94.42 / 0.326 $&$ 94.82 / 0.195 $&$ 94.66 / 0.113 $ \\
            &$(99,100)$&$ 94.92 / 0.531 $&$ 94.24 / 0.261 $&$ 94.66 / 0.125 $&$ 93.62 / 0.058 $ \\
            &$(1,100)$ &$ 94.06 / 0.530 $&$ 94.24 / 0.339 $&$ 94.62 / 0.227 $&$ 94.68 / 0.157 $ \\
            &$(1,50)$ &$ 94.8 / 0.530 $&$ 93.92 / 0.370 $&$ 94.54 / 0.255 $&$ 95.12 / 0.173 $\\
            &$(0,0)^*$ &$ 94.68 / 0.039 $&$ 94.64 / 0.024 $&$ 95.06 / 0.013 $&$ 94.94 / 0.007 $ \\

&&&&&&\\
200         &$(1,2)   $&$ 94.34 / 0.401  $&$ 94.4 / 0.301  $&$ 94.64 / 0.215  $&$ 95.32 / 0.148  $ \\
            &$(50,51) $&$ 94.34 / 0.402  $&$ 94.64 / 0.233  $&$ 94.98 / 0.126  $&$ 94.56 / 0.068  $ \\
            &$(99,100)$&$ 94.66 / 0.406  $&$ 94.6 / 0.177  $&$ 93.66 / 0.075  $&$ 94.26 / 0.031  $\\
            &$(1,100)$ &$ 94.5 / 0.404  $&$ 94.68 / 0.242  $&$ 94.86 / 0.155  $&$ 95.26 / 0.102  $ \\
            &$(1,50)$ &$ 94.04 / 0.402  $&$ 94.16 / 0.267  $&$ 94.58 / 0.173  $&$ 94.98 / 0.112  $ \\
            &$(0,0)$ &$ 95.1 / 0.020  $&$ 95.26 / 0.011  $&$ 95.2 / 0.006  $&$ 94.9 / 0.003  $ \\
\hline
\end{tabular}
\begin{tablenotes}
        \footnotesize
        \item[$^{*}$] $(0,0)$ indicates the simulated results for $\widehat{\mu}$.
      \end{tablenotes}
    \end{threeparttable}
\end{table}
}

Table \ref{Table:gamma} reports simulation results for the estimate $\bs{\widehat{\gamma}}$
and bias correction estimate  $\bs{\widehat{\gamma}}_{bc} (= \bs{\widehat{\gamma}}- \hat{I}^{-1} \hat{B}/\sqrt{n(n-1)})$ at the nominal level $95\%$.
The coverage frequencies for the uncorrected estimate are all close to those for corrected estimates.
All simulated coverage frequencies achieves the nominal level.
As expected, the standard error increases with $c$ and decreases with $n$.

{\renewcommand{\arraystretch}{1}
\begin{table}[!htbp]\centering
\caption{
The reported values are the coverage frequency ($\times 100\%$) of $\gamma_i$ for $i$ with corrected estimates (uncorrected estimates) / length of confidence interval
 /the frequency ($\times 100\%$) that the MLE did not exist ($\bs{\gamma}^*=(1, 1, 1)^\top$).
}
\label{Table:gamma}
\begin{tabular}{cclllcc}
\hline
$n$     &   $c$            & $\gamma_1$            & $\gamma_2$ & $\gamma_3$\\
\hline
$100$   & $0$             &$ 94.68 ( 94.66 ) / 0.039 $&$ 94.7 ( 94.86 ) / 0.248 $&$ 95.42 ( 95.46 ) / 0.066 $ \\

        & $0.2$           &$ 94.64 ( 94.64 ) / 0.024 $&$ 94.04 ( 94.14 ) / 0.151 $&$ 94.56 ( 94.56 ) / 0.040 $ \\

        & $0.4$           &$ 95.06 ( 95.06 ) / 0.013 $&$ 94.5 ( 94.58 ) / 0.086 $&$ 94.72 ( 94.74 ) / 0.023 $ \\

        & $0.6$           &$ 94.94 ( 94.94 ) / 0.007 $&$ 95.44 ( 95.5 ) / 0.047 $&$ 94.7 ( 94.7 ) / 0.012 $\\

$200$   & $0$            &$ 95.1 ( 95.04 ) / 0.020 $&$ 94.62 ( 94.8 ) / 0.132 $&$ 95.18 ( 95.2 ) / 0.036 $ \\
        & $0.2$      &$ 95.26 ( 95.26 ) / 0.011 $&$ 94.76 ( 94.96 ) / 0.074 $&$ 95.6 ( 95.62 ) / 0.020 $\\

        & $0.4$            &$ 95.2 ( 95.2 ) / 0.006 $&$ 95.08 ( 95.08 ) / 0.038 $&$ 94.68 ( 94.68 ) / 0.010 $ \\
        & $0.6$      &$ 94.9 ( 94.9 ) / 0.003 $&$ 95.06 ( 95.04 ) / 0.018 $&$ 94.52 ( 94.52 ) / 0.005 $\\

\hline
\end{tabular}
\end{table}
}

\subsection{A real data example}

We use the Enron email dataset as an example analysis [\cite{Cohen2004}], available from \url{https://www.cs.cmu.edu/~enron/}.
The raw data is messy and needs to be cleaned before any analysis is conducted.
\cite{Zhou2007} applied data cleaning strategies to compile the Enron email dataset.
We use their cleaned data for the subsequent analysis.
A weighted directed graph formed, where nodes represent users and an weighted edge denotes that there are $k$ messages from user $i$ to user $j$
if  $a_{ij}=k$.
This leaves a strongly connected network with $156$ nodes and $2715$ edges having positive weights.
The minimum, $1/4$ quantile, median, $3/4$ quantile and maximum values of $d$ are $0$, $18.75$, $80$, $307$ and $3844$, respectively. And the minimum, $1/4$ quantile, median, $3/4$ quantile and maximum values of $b$ are $0$, $70.75$, $142$, $299.25$ and $1797$, respectively.
\begin{figure}[!htp]
  \centering
  \subfigure{
  \centering
    \includegraphics[width=0.45\linewidth]{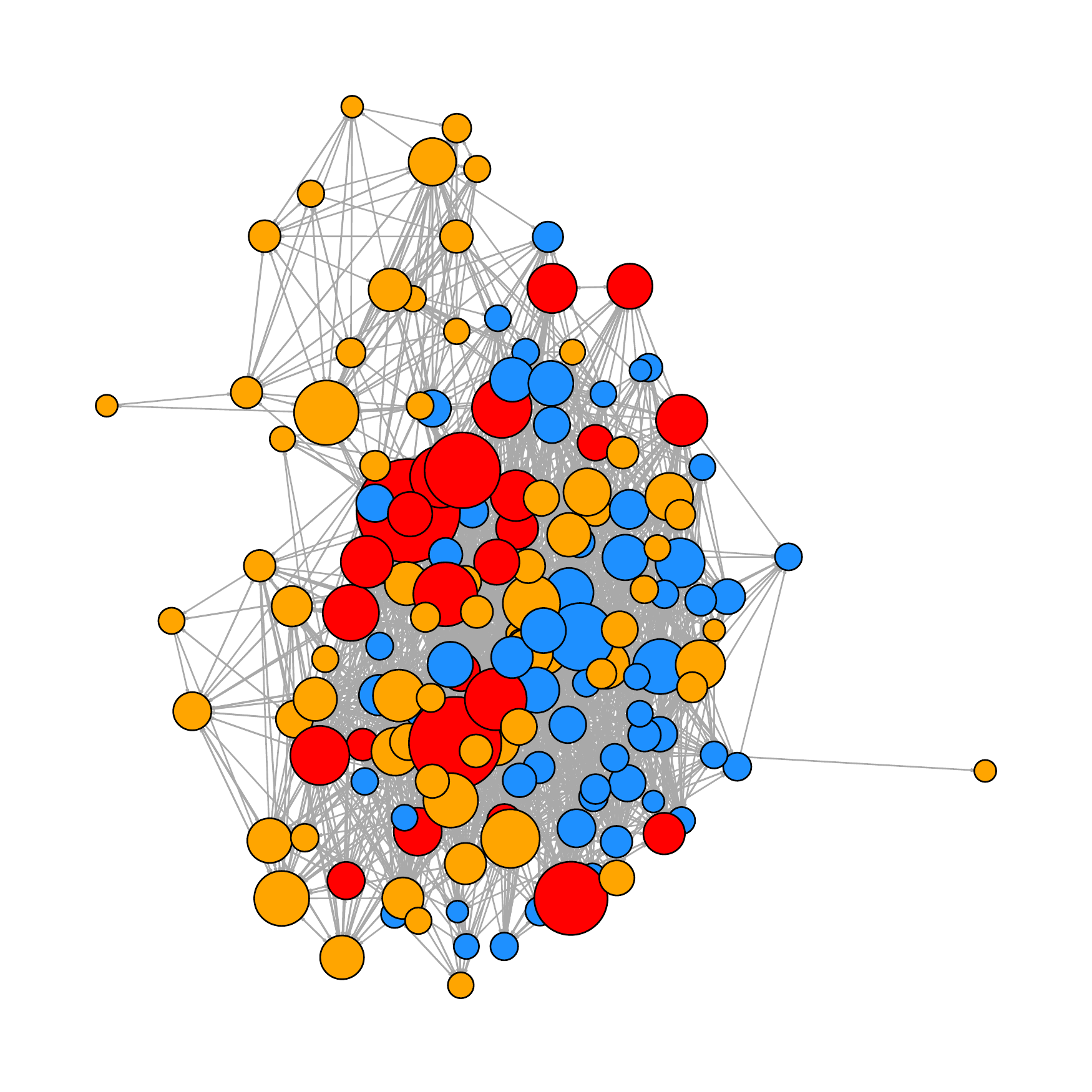}}
  \subfigure{
    \centering
    \includegraphics[width=0.45\linewidth]{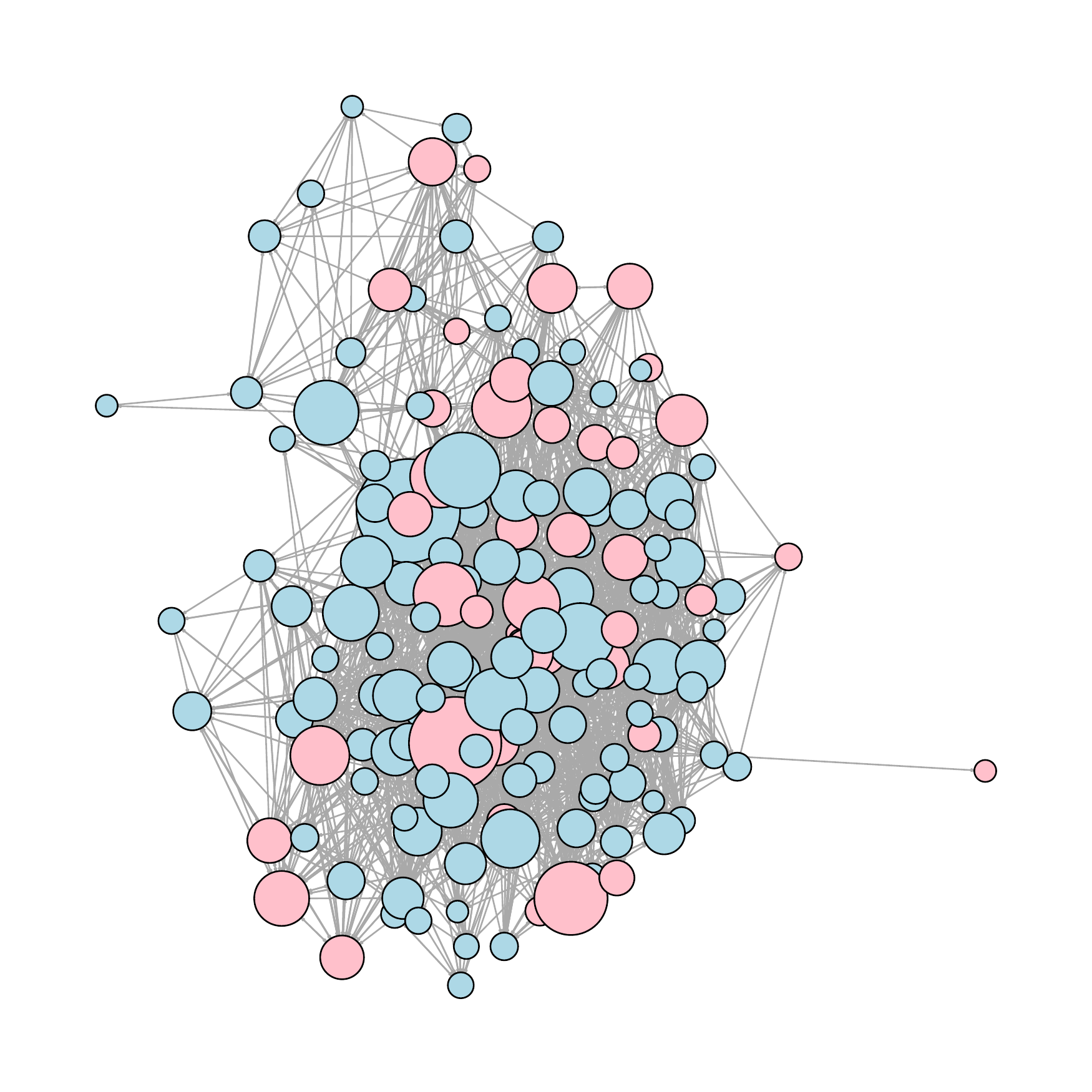}}
   \caption{Visualization of Enran email network. 
The vertex sizes are proportional to nodal degrees, where the sizes with degrees less than $5$ are the same.
In the left graph, the colors indicate different departments (red for legal and blue for trading and orange for other), while in the right graph, the colors represent different genders (light blue for male and pink for female).
}
\label{figure-data} 
\end{figure}

\begin{table}[!h]\centering
\renewcommand\arraystretch{1.0}
\scriptsize
\vskip5pt
\caption{The estimates of $\alpha_i$ and $\beta_j$ and their standard errors in the Enran email data set
with $\widehat{\alpha}_{1}=\widehat{\beta}_{1}=0$.}
\vskip5pt
\label{Table:alphabeta:real}
\resizebox{\textwidth}{11cm}{
\begin{tabular}{ccc ccc ccc ccc ccc}
\hline
Node & $d_i$  &  $\hat{\alpha}_i$ & $\hat{\sigma}_i$ & $b_j$ &  $\hat{\beta}_i$ & $\hat{\sigma}_j$  && Node & $d_i$  &  $\hat{\alpha}_i$ & $\hat{\sigma}_i$ & $b_j$ &  $\hat{\beta}_i$ & $\hat{\sigma}_j$ \\
\hline
\\
2 &$ 18 $&$ -1.884 $&$ 0.241 $&$ 70 $&$ -1.627 $&$ 0.131 $&& 80 &$ 79 $&$ -0.520 $&$ 0.123 $&$ 167 $&$ -0.894 $&$ 0.094 $ \\
3 &$ 15 $&$ -1.787 $&$ 0.264 $&$ 304 $&$ -0.043 $&$ 0.079 $&& 81 &$ 132 $&$ -0.003 $&$ 0.100 $&$ 226 $&$ -0.588 $&$ 0.086 $ \\
4 &$ 312 $&$ 1.210 $&$ 0.076 $&$ 904 $&$ 0.518 $&$ 0.064 $&& 82 &$ 220 $&$ 0.504 $&$ 0.084 $&$ 152 $&$ -0.980 $&$ 0.098 $ \\
5 &$ 428 $&$ 1.563 $&$ 0.069 $&$ 277 $&$ -0.091 $&$ 0.081 $&& 83 &$ 10 $&$ -2.472 $&$ 0.320 $&$ 68 $&$ -1.657 $&$ 0.133 $ \\
6 &$ 204 $&$ 0.797 $&$ 0.086 $&$ 50 $&$ -1.828 $&$ 0.152 $&& 84 &$ 55 $&$ -0.766 $&$ 0.144 $&$ 96 $&$ -1.309 $&$ 0.116 $ \\
7 &$ 336 $&$ 1.105 $&$ 0.074 $&$ 139 $&$ -0.977 $&$ 0.101 $&& 85 &$ 374 $&$ 1.184 $&$ 0.072 $&$ 161 $&$ -0.804 $&$ 0.096 $ \\
8 &$ 2 $&$ -4.084 $&$ 0.709 $&$ 45 $&$ -2.070 $&$ 0.159 $&& 86 &$ 1 $&$ -4.750 $&$ 1.002 $&$ 33 $&$ -2.416 $&$ 0.182 $ \\
9 &$ 296 $&$ 0.979 $&$ 0.076 $&$ 141 $&$ -0.966 $&$ 0.100 $&& 87 &$ 58 $&$ -0.681 $&$ 0.141 $&$ 141 $&$ -0.960 $&$ 0.100 $ \\
10 &$ 29 $&$ -1.152 $&$ 0.192 $&$ 82 $&$ -1.279 $&$ 0.123 $&& 88 &$ 29 $&$ -1.156 $&$ 0.192 $&$ 50 $&$ -1.773 $&$ 0.151 $ \\
11 &$ 61 $&$ -0.654 $&$ 0.138 $&$ 203 $&$ -0.559 $&$ 0.089 $&& 89 &$ 10 $&$ -2.472 $&$ 0.320 $&$ 68 $&$ -1.657 $&$ 0.133 $ \\
12 &$ 11 $&$ -2.353 $&$ 0.306 $&$ 20 $&$ -2.916 $&$ 0.230 $&& 90 &$ 317 $&$ 0.994 $&$ 0.075 $&$ 206 $&$ -0.526 $&$ 0.088 $ \\
13 &$ 12 $&$ -2.262 $&$ 0.293 $&$ 75 $&$ -1.595 $&$ 0.128 $&& 91 &$ 480 $&$ 1.616 $&$ 0.067 $&$ 366 $&$ -0.257 $&$ 0.075 $ \\
14 &$ 14 $&$ -1.883 $&$ 0.272 $&$ 47 $&$ -1.910 $&$ 0.156 $&& 92 &$ 822 $&$ 2.158 $&$ 0.061 $&$ 532 $&$ 0.018 $&$ 0.070 $ \\
15 &$ 93 $&$ -0.048 $&$ 0.115 $&$ 96 $&$ -1.737 $&$ 0.116 $&& 93 &$ 19 $&$ -1.574 $&$ 0.235 $&$ 80 $&$ -1.377 $&$ 0.124 $ \\
16 &$ 9 $&$ -2.419 $&$ 0.337 $&$ 18 $&$ -2.993 $&$ 0.242 $&& 94 &$ 20 $&$ -1.370 $&$ 0.229 $&$ 48 $&$ -1.800 $&$ 0.154 $ \\
17 &$ 396 $&$ 1.291 $&$ 0.071 $&$ 425 $&$ 0.146 $&$ 0.073 $&& 95 &$ 501 $&$ 1.641 $&$ 0.067 $&$ 183 $&$ -1.069 $&$ 0.092 $ \\
18 &$ 42 $&$ -1.031 $&$ 0.162 $&$ 153 $&$ -0.844 $&$ 0.097 $&& 96 &$ 10 $&$ -2.447 $&$ 0.320 $&$ 31 $&$ -2.478 $&$ 0.188 $ \\
19 &$ 97 $&$ 0.052 $&$ 0.113 $&$ 42 $&$ -2.014 $&$ 0.164 $&& 97 &$ 8 $&$ -2.670 $&$ 0.357 $&$ 33 $&$ -2.416 $&$ 0.182 $ \\
20 &$ 3844 $&$ 3.464 $&$ 0.052 $&$ 861 $&$ 0.623 $&$ 0.065 $&& 98 &$ 319 $&$ 1.009 $&$ 0.075 $&$ 318 $&$ -0.092 $&$ 0.078 $ \\
21 &$ 1 $&$ -4.616 $&$ 1.001 $&$ 19 $&$ -2.939 $&$ 0.236 $&& 99 &$ 6 $&$ -2.731 $&$ 0.411 $&$ 37 $&$ -2.150 $&$ 0.173 $ \\
22 &$ 16 $&$ -1.751 $&$ 0.255 $&$ 32 $&$ -2.294 $&$ 0.185 $&& 100 &$ 108 $&$ -0.052 $&$ 0.108 $&$ 252 $&$ -0.376 $&$ 0.083 $ \\
23 &$ 5 $&$ -3.133 $&$ 0.451 $&$ 131 $&$ -1.037 $&$ 0.103 $&& 101 &$ 138 $&$ 0.350 $&$ 0.098 $&$ 85 $&$ -1.740 $&$ 0.121 $ \\
24 &$ 518 $&$ 1.373 $&$ 0.066 $&$ 380 $&$ -0.046 $&$ 0.075 $&& 102 &$ 37 $&$ -1.125 $&$ 0.172 $&$ 230 $&$ -0.472 $&$ 0.085 $ \\
25 &$ 60 $&$ -0.738 $&$ 0.138 $&$ 85 $&$ -1.880 $&$ 0.121 $&& 103 &$ 7 $&$ -2.802 $&$ 0.381 $&$ 53 $&$ -1.942 $&$ 0.148 $ \\
26 &$ 111 $&$ 0.156 $&$ 0.107 $&$ 565 $&$ 0.036 $&$ 0.069 $&& 104 &$ 49 $&$ -0.631 $&$ 0.151 $&$ 56 $&$ -1.658 $&$ 0.144 $ \\
27 &$ 305 $&$ 1.132 $&$ 0.076 $&$ 341 $&$ -0.026 $&$ 0.077 $&& 105 &$ 7 $&$ -2.803 $&$ 0.381 $&$ 46 $&$ -2.084 $&$ 0.157 $ \\
28 &$ 7 $&$ -2.579 $&$ 0.381 $&$ 17 $&$ -2.928 $&$ 0.249 $&& 106 &$ 56 $&$ -0.591 $&$ 0.143 $&$ 19 $&$ -2.935 $&$ 0.236 $ \\
29 &$ 123 $&$ 0.293 $&$ 0.103 $&$ 71 $&$ -1.486 $&$ 0.131 $&& 107 &$ 25 $&$ -1.530 $&$ 0.206 $&$ 45 $&$ -2.104 $&$ 0.159 $ \\
30 &$ 15 $&$ -2.061 $&$ 0.263 $&$ 149 $&$ -0.872 $&$ 0.098 $&& 108 &$ 15 $&$ -1.812 $&$ 0.263 $&$ 60 $&$ -1.666 $&$ 0.140 $ \\
31 &$ 22 $&$ -1.686 $&$ 0.219 $&$ 41 $&$ -2.162 $&$ 0.165 $&& 109 &$ 294 $&$ 1.045 $&$ 0.077 $&$ 560 $&$ 0.091 $&$ 0.069 $ \\
32 &$ 0 $&$ -15.218 $&$ 200.343 $&$ 0 $&$ -15.645 $&$ 131.701 $&& 110 &$ 9 $&$ -2.548 $&$ 0.337 $&$ 99 $&$ -1.317 $&$ 0.114 $ \\
33 &$ 144 $&$ 0.197 $&$ 0.097 $&$ 102 $&$ -1.242 $&$ 0.113 $&& 111 &$ 28 $&$ -1.415 $&$ 0.195 $&$ 61 $&$ -1.800 $&$ 0.139 $ \\
34 &$ 676 $&$ 1.695 $&$ 0.063 $&$ 329 $&$ -0.499 $&$ 0.077 $&& 112 &$ 119 $&$ 0.182 $&$ 0.104 $&$ 247 $&$ -0.364 $&$ 0.084 $ \\
35 &$ 12 $&$ -2.126 $&$ 0.293 $&$ 84 $&$ -1.452 $&$ 0.122 $&& 113 &$ 316 $&$ 0.948 $&$ 0.075 $&$ 642 $&$ 0.154 $&$ 0.067 $ \\
36 &$ 39 $&$ -1.079 $&$ 0.168 $&$ 137 $&$ -0.990 $&$ 0.101 $&& 114 &$ 126 $&$ 0.324 $&$ 0.102 $&$ 146 $&$ -0.691 $&$ 0.099 $ \\
37 &$ 0 $&$ -15.377 $&$ 202.821 $&$ 0 $&$ -15.726 $&$ 135.152 $&& 115 &$ 151 $&$ 0.272 $&$ 0.095 $&$ 90 $&$ -1.402 $&$ 0.119 $ \\
38 &$ 274 $&$ 1.020 $&$ 0.078 $&$ 290 $&$ -0.190 $&$ 0.080 $&& 116 &$ 54 $&$ -0.531 $&$ 0.145 $&$ 60 $&$ -1.662 $&$ 0.140 $ \\
39 &$ 628 $&$ 1.926 $&$ 0.064 $&$ 93 $&$ -1.161 $&$ 0.118 $&& 117 &$ 797 $&$ 2.119 $&$ 0.061 $&$ 409 $&$ -0.246 $&$ 0.074 $ \\
40 &$ 78 $&$ -0.412 $&$ 0.124 $&$ 151 $&$ -0.854 $&$ 0.098 $&& 118 &$ 273 $&$ 1.011 $&$ 0.078 $&$ 229 $&$ -0.427 $&$ 0.086 $ \\
41 &$ 169 $&$ 0.618 $&$ 0.092 $&$ 143 $&$ -0.781 $&$ 0.100 $&& 119 &$ 251 $&$ 0.918 $&$ 0.080 $&$ 120 $&$ -1.075 $&$ 0.106 $ \\
42 &$ 151 $&$ 0.504 $&$ 0.095 $&$ 131 $&$ -0.870 $&$ 0.103 $&& 120 &$ 1515 $&$ 2.813 $&$ 0.056 $&$ 1173 $&$ 0.857 $&$ 0.062 $ \\
43 &$ 1204 $&$ 2.342 $&$ 0.057 $&$ 359 $&$ 0.099 $&$ 0.076 $&& 121 &$ 269 $&$ 0.714 $&$ 0.079 $&$ 314 $&$ -0.252 $&$ 0.078 $ \\
44 &$ 69 $&$ -0.506 $&$ 0.130 $&$ 166 $&$ -0.796 $&$ 0.095 $&& 122 &$ 522 $&$ 1.501 $&$ 0.066 $&$ 1730 $&$ 1.155 $&$ 0.059 $ \\
45 &$ 393 $&$ 1.156 $&$ 0.071 $&$ 417 $&$ -0.275 $&$ 0.073 $&& 123 &$ 126 $&$ 0.075 $&$ 0.102 $&$ 259 $&$ -0.311 $&$ 0.082 $ \\
46 &$ 905 $&$ 2.155 $&$ 0.060 $&$ 292 $&$ -0.526 $&$ 0.080 $&& 124 &$ 64 $&$ -0.735 $&$ 0.135 $&$ 87 $&$ -1.547 $&$ 0.120 $ \\
47 &$ 197 $&$ 0.431 $&$ 0.087 $&$ 809 $&$ 0.691 $&$ 0.065 $&& 125 &$ 53 $&$ -0.771 $&$ 0.146 $&$ 143 $&$ -0.946 $&$ 0.100 $ \\
48 &$ 396 $&$ 1.100 $&$ 0.071 $&$ 303 $&$ -0.280 $&$ 0.079 $&& 126 &$ 42 $&$ -1.005 $&$ 0.162 $&$ 130 $&$ -1.042 $&$ 0.103 $ \\
49 &$ 434 $&$ 1.534 $&$ 0.069 $&$ 796 $&$ 0.398 $&$ 0.065 $&& 127 &$ 33 $&$ -1.245 $&$ 0.181 $&$ 148 $&$ -0.913 $&$ 0.098 $ \\
50 &$ 10 $&$ -2.468 $&$ 0.321 $&$ 135 $&$ -0.971 $&$ 0.102 $&& 128 &$ 10 $&$ -2.468 $&$ 0.321 $&$ 134 $&$ -0.979 $&$ 0.102 $ \\
51 &$ 16 $&$ -1.745 $&$ 0.255 $&$ 91 $&$ -1.249 $&$ 0.118 $&& 129 &$ 100 $&$ -0.005 $&$ 0.112 $&$ 85 $&$ -1.433 $&$ 0.121 $ \\
52 &$ 0 $&$ -15.218 $&$ 200.353 $&$ 3 $&$ -4.785 $&$ 0.580 $&& 130 &$ 1019 $&$ 2.304 $&$ 0.059 $&$ 793 $&$ 0.482 $&$ 0.065 $ \\
53 &$ 162 $&$ 0.296 $&$ 0.093 $&$ 1012 $&$ 0.602 $&$ 0.063 $&& 131 &$ 1666 $&$ 2.670 $&$ 0.055 $&$ 1797 $&$ 1.250 $&$ 0.059 $ \\
54 &$ 22 $&$ -1.685 $&$ 0.219 $&$ 50 $&$ -1.963 $&$ 0.151 $&& 132 &$ 12 $&$ -2.261 $&$ 0.293 $&$ 88 $&$ -1.435 $&$ 0.120 $ \\
55 &$ 16 $&$ -1.989 $&$ 0.256 $&$ 255 $&$ -0.335 $&$ 0.083 $&& 133 &$ 16 $&$ -1.976 $&$ 0.255 $&$ 42 $&$ -2.174 $&$ 0.164 $ \\
56 &$ 128 $&$ -0.038 $&$ 0.101 $&$ 156 $&$ -0.959 $&$ 0.097 $&& 134 &$ 22 $&$ -1.682 $&$ 0.219 $&$ 91 $&$ -1.365 $&$ 0.118 $ \\
57 &$ 266 $&$ 0.701 $&$ 0.079 $&$ 291 $&$ -0.328 $&$ 0.080 $&& 135 &$ 81 $&$ -0.375 $&$ 0.122 $&$ 140 $&$ -0.930 $&$ 0.100 $ \\
58 &$ 221 $&$ 0.829 $&$ 0.084 $&$ 206 $&$ -0.849 $&$ 0.088 $&& 136 &$ 10 $&$ -2.471 $&$ 0.320 $&$ 89 $&$ -1.388 $&$ 0.119 $ \\
59 &$ 2858 $&$ 3.432 $&$ 0.053 $&$ 864 $&$ 0.642 $&$ 0.065 $&& 137 &$ 377 $&$ 1.341 $&$ 0.072 $&$ 300 $&$ -0.147 $&$ 0.079 $ \\
60 &$ 218 $&$ 0.491 $&$ 0.084 $&$ 86 $&$ -1.550 $&$ 0.121 $&& 138 &$ 918 $&$ 2.055 $&$ 0.060 $&$ 1477 $&$ 1.016 $&$ 0.060 $ \\
61 &$ 613 $&$ 1.596 $&$ 0.064 $&$ 1290 $&$ 1.183 $&$ 0.061 $&& 139 &$ 0 $&$ -15.247 $&$ 212.857 $&$ 6 $&$ -3.970 $&$ 0.412 $ \\
62 &$ 28 $&$ -1.189 $&$ 0.196 $&$ 55 $&$ -1.751 $&$ 0.145 $&& 140 &$ 119 $&$ 0.271 $&$ 0.104 $&$ 177 $&$ -0.499 $&$ 0.093 $ \\
63 &$ 321 $&$ 1.429 $&$ 0.075 $&$ 210 $&$ -0.283 $&$ 0.088 $&& 141 &$ 20 $&$ -1.754 $&$ 0.229 $&$ 37 $&$ -2.301 $&$ 0.173 $ \\
64 &$ 10 $&$ -2.218 $&$ 0.320 $&$ 63 $&$ -1.617 $&$ 0.137 $&& 142 &$ 15 $&$ -1.632 $&$ 0.264 $&$ 236 $&$ -0.208 $&$ 0.085 $ \\
65 &$ 723 $&$ 1.915 $&$ 0.062 $&$ 688 $&$ 0.654 $&$ 0.066 $&& 143 &$ 225 $&$ 0.672 $&$ 0.083 $&$ 483 $&$ 0.320 $&$ 0.071 $ \\
66 &$ 67 $&$ -0.143 $&$ 0.132 $&$ 178 $&$ -0.483 $&$ 0.092 $&& 144 &$ 286 $&$ 1.320 $&$ 0.077 $&$ 256 $&$ -0.089 $&$ 0.083 $ \\
67 &$ 775 $&$ 1.779 $&$ 0.061 $&$ 422 $&$ 0.074 $&$ 0.073 $&& 145 &$ 627 $&$ 1.754 $&$ 0.064 $&$ 462 $&$ 0.247 $&$ 0.072 $ \\
68 &$ 85 $&$ -0.450 $&$ 0.119 $&$ 104 $&$ -1.367 $&$ 0.112 $&& 146 &$ 38 $&$ -1.112 $&$ 0.170 $&$ 31 $&$ -2.476 $&$ 0.188 $ \\
69 &$ 383 $&$ 1.226 $&$ 0.071 $&$ 412 $&$ 0.136 $&$ 0.073 $&& 147 &$ 71 $&$ -0.614 $&$ 0.129 $&$ 387 $&$ -0.054 $&$ 0.074 $ \\
70 &$ 5 $&$ -3.168 $&$ 0.450 $&$ 37 $&$ -2.266 $&$ 0.173 $&& 148 &$ 62 $&$ -0.587 $&$ 0.136 $&$ 112 $&$ -1.213 $&$ 0.109 $ \\
71 &$ 13 $&$ -2.178 $&$ 0.282 $&$ 124 $&$ -1.092 $&$ 0.105 $&& 149 &$ 307 $&$ 0.968 $&$ 0.076 $&$ 285 $&$ -0.202 $&$ 0.080 $ \\
72 &$ 416 $&$ 1.337 $&$ 0.070 $&$ 374 $&$ 0.019 $&$ 0.075 $&& 150 &$ 5 $&$ -3.140 $&$ 0.450 $&$ 37 $&$ -2.302 $&$ 0.173 $ \\
73 &$ 55 $&$ -0.735 $&$ 0.144 $&$ 133 $&$ -1.019 $&$ 0.102 $&& 151 &$ 1045 $&$ 2.223 $&$ 0.058 $&$ 299 $&$ -0.134 $&$ 0.080 $ \\
74 &$ 436 $&$ 1.350 $&$ 0.069 $&$ 334 $&$ -0.070 $&$ 0.077 $&& 152 &$ 20 $&$ -1.744 $&$ 0.229 $&$ 171 $&$ -0.770 $&$ 0.094 $ \\
75 &$ 101 $&$ -0.149 $&$ 0.111 $&$ 215 $&$ -0.499 $&$ 0.087 $&& 153 &$ 40 $&$ -1.200 $&$ 0.166 $&$ 179 $&$ -0.827 $&$ 0.092 $ \\
76 &$ 19 $&$ -1.823 $&$ 0.235 $&$ 175 $&$ -0.711 $&$ 0.093 $&& 154 &$ 117 $&$ -0.125 $&$ 0.105 $&$ 199 $&$ -0.716 $&$ 0.089 $ \\
77 &$ 235 $&$ 0.798 $&$ 0.082 $&$ 137 $&$ -1.321 $&$ 0.101 $&& 155 &$ 33 $&$ -1.398 $&$ 0.181 $&$ 72 $&$ -1.738 $&$ 0.130 $ \\
78 &$ 44 $&$ -0.986 $&$ 0.159 $&$ 130 $&$ -1.006 $&$ 0.103 $&& 156 &$ 307 $&$ 0.972 $&$ 0.076 $&$ 338 $&$ -0.031 $&$ 0.077 $ \\
\hline
\end{tabular}}

\end{table}

Each employee has three categorical variables: departments of these employees (Trading, Legal, Other),
the genders (Male, Female) and seniorities (Senior, Junior). 
We plot the network with individual departments and genders in Figure \ref{figure-data}.
We can see that the degrees exhibit a great variation across nodes and
it is not easy to judge homophic or heteriphic effects that require quantitative analysis.
The $3$-dimensional covariate vector $x_{ij}$ of edge $(i,j)$ is formed by using a homophilic matching function between these three covariates of two employees $i$ and $j$, i.e.,
if the $k$th attributes of $i$ and $j$ are equal, then $z_{ijk}=1$; otherwise $z_{ijk}=-1$.

{\renewcommand{\arraystretch}{1}
\begin{table}[!htp]\centering
\caption{The MLE $\widehat{\gamma}$ and $p$-values fitted in the network Poisson model.}
\label{Table:gamma:realdata}
\vskip5pt
\begin{tabular}{lcccl}
\hline
Covariate       &  $\hat{\gamma}$ & $\hat{\gamma}_{bc}$ & Standard Error & $p$-value   \\
\hline
Department      &$ 0.765 $& $0.764  $ & $ 5.85 \times 10^{-3}$ & $<10^{-3}$\\
Gender          &$ 0.223 $& $0.221 $& $ 5.82 \times 10^{-3}$ & $<10^{-3}$\\
Seniority       &$ 0.405$& $0.403  $   &$ 5.76 \times 10^{-3}$  & $<10^{-3}$\\
\hline
\end{tabular}
\end{table}

The estimates of $\alpha_i$ and $\beta_j$ with their estimated standard errors are given in Table 3, which
 vary from the minimum $ -15.377$ to maximum $3.464$ and the minimum $ -15.726$ to maximum $1.250$,  respectively.
The estimated covariate effects, their  bias corrected estimates, their standard errors, and their $p$-values
with bias-correction under the null of having no effects are reported in Table \ref{Table:gamma:realdata}.
The variables department, gender and seniority do have significant influence on the formation of organizational emails.
This is in sharp contrast to fitted results in Table \ref{Table:gamma:sparsep0} by using the covariate-$p_0$ model
 with the weighted values being  simply treated as unweighted values, in which all covariates are not significant.
The estimate $\widehat{\mu}$ of the density parameter $\mu$ is $1.269$ with the standard error $0.070$
and $p$-value less than $10^{-3}$ under the null $\mu=0$.

{\renewcommand{\arraystretch}{1}
\begin{table}[!htp]\centering
\caption{The MLE $\widehat{\gamma}$ and $p$-values fitted in sparse $p_0$-model.}
\label{Table:gamma:sparsep0}
\vskip5pt
\begin{tabular}{lcccl}
\hline
Covariate       &  $\hat{\gamma}$ & $\hat{\gamma}_{bc}$ & Standard Error & $p$-value   \\
\hline
Department      &$ -1.64  \times 10^{-2}$& $-1.78  \times 10^{-2}$ & $ 2.95 \times 10^{-2}$ & $0.55$\\
Gender          &$ 6.30 \times 10^{-2} $& $3.83  \times 10^{-2}$& $ 3.33 \times 10^{-2}$ & $0.25$\\
Seniority       &$ 3.16  \times 10^{-2}$& $3.47   \times 10^{-2}$   &$ 2.72 \times 10^{-2}$  & $0.20$\\
\hline
\end{tabular}
\end{table}
{\renewcommand{\arraystretch}{1}

\section{Discussion}
\label{section:summary}

In this paper, we have derived the $\ell_\infty$-error between the MLE and
established its asymptotic normality in the network Poisson model for weighted directed networks.
Note that the conditions imposed on $\rho_n$
in Theorems \ref{Theorem:con}--\ref{theorem:covariate:asym} may not be best possible.
In particular, the conditions guaranteeing the asymptotic normality seem stronger than those guaranteeing the consistency.
It would be of interest to investigate whether these conditions can be relaxed.

We do not consider the reciprocity parameter in our model.
As discussed in \cite{Yan-Jiang-Fienberg-Leng2018},
there is an implicit taste for the reciprocity effect, although we do not include this parameter.
Since there is a tendency toward reciprocity among nodes sharing similar node features,
it would alleviate the lack of a reciprocity term to some extent.
To measure the reciprocity of dyads, it is natural to model the distribution of the dyad $(a_{ij}, a_{ji})$
by using the bivariate Poisson distribution.
Developing a new model is generally relatively easy while
 the problem of investigating the asymptotic
theory of the MLE becomes more challenging.
In particular, the Fisher information matrix for the parameter vector $(\rho, \alpha_1,\ldots,\alpha_n, \beta_1, \ldots, \beta_{n-1})$ is not diagonally dominant and thus does not belong to the matrix class $\mathcal{L}_{n}(m, M)$.
In order to make extensions, a new approximate matrix for approximating the inverse of the Fisher information matrix is needed.
This is beyond of the present paper.

\renewcommand{\baselinestretch}{1.2}\selectfont

\section{Appendix: Proofs for theorems}
\label{section:proofs}
We only give the proof of Theorem \ref{Theorem:con} here.
The proof of Theorem \ref{Theorem:binary:central} is put in the supplementary material.

Let $F(\mathbf{x}): \R^n \to \R^n$ be a function vector on $\mathbf{x}\in\R^n$. We say that a Jacobian matrix $F^\prime(\mathbf{x)}$ with $\mathbf{x}\in \R^n$ is Lipschitz continuous on a convex set $D\subset\R^n$ if
for any $\mathbf{x},\mathbf{y}\in D$, there exists a constant $L>0$ such that
for any vector $\mathbf{v}\in \R^n$ the inequality
\begin{equation*}
\| [F^\prime (\mathbf{x})] \mathbf{v} - [F^\prime (\mathbf{y})] \mathbf{v} \|_\infty \le L \| \mathbf{x} - \mathbf{y} \|_\infty \|\mathbf{v}\|_\infty
\end{equation*}
holds.
We will use the Newton iterative sequence to establish the existence and consistency of the moment estimator.
\cite{Gragg:Tapia:1974} gave the optimal error bound for the Newton method under the Kantovorich conditions
[\cite{Kantorovich1948Functional}].

\begin{lemma}[\cite{Gragg:Tapia:1974}]\label{pro:Newton:Kantovorich}
Let $D$ be an open convex set of $\R^n$ and $F:D \to \R^n$ a differential function
with a Jacobian $F^\prime(\mathbf{x})$ that is Lipschitz continuous on $D$ with Lipschitz coefficient $\lambda$.
Assume that $\mathbf{x}_0 \in D$ is such that $[ F^\prime (\mathbf{x}_0) ]^{-1} $ exists,
\begin{eqnarray*}
\| [ F^\prime (\mathbf{x}_0 ) ]^{-1} \|_\infty  \le \aleph,~~ \| [ F^\prime (\mathbf{x}_0) ]^{-1} F(\mathbf{x}_0) \|_\infty \le \delta, ~~ \rho= 2 \aleph L \delta \le 1,
\\
B(\mathbf{x}_0, t^*) \subset D, ~~ t^* = \frac{2}{\rho} ( 1 - \sqrt{1-\rho} ) \delta = \frac{ 2\delta }{ 1 + \sqrt{1-\rho} }\le 2\delta.
\end{eqnarray*}
Then: (1) The Newton iterations $\mathbf{x}_{k+1} = \mathbf{x}_k - [ F^\prime (\mathbf{x}_k) ]^{-1} F(\mathbf{x}_k)$ exist and $\mathbf{x}_k \in B(\mathbf{x}_0, t^*) \subset D$ for $k \ge 0$. (2)
$\mathbf{x}^* = \lim \mathbf{x}_k$ exists, $\mathbf{x}^* \in \overline{ B(\mathbf{x}_0, t^*) } \subset D$ and $F(\mathbf{x}^*)=0$.
\end{lemma}

\subsection{Proof of Theorem \ref{Theorem:con}}
\label{subsection:Theorem:con}

To show Theorem \ref{Theorem:con}, we need three lemmas below.

\begin{lemma}\label{lemma-Q-Lip}
Let $D=(B(\bs{\gamma}^*, \epsilon_{n2}))^p (\subset \R^{p})$.
If $\| F(\bs{\theta}^*, \bs{\gamma}^*) \|_\infty = O( (n\log n)^{1/2} )$, then
$ Q_c(\bs{\gamma})$ is Lipschitz continuous on $D$ with the Lipschitz coefficient  $O(n^2e^{19\rho_n})$.
\end{lemma}

\begin{lemma}\label{lemma-diff-F-Q}
With probability at least $1-O(1/n)$, we have
\begin{equation}
\| F(\bs{\theta}^*, \bs{\gamma}^*) \|_\infty \lesssim e^{2\rho_n}(n\log{n})^{1/2} , ~~\| Q(\bs{\theta}^*, \bs{\gamma}^*) \|_\infty \lesssim qe^{2\rho_n}n(\log{n})^{1/2},
\end{equation}
where $q:=\max_{i,j} \|Z_{ij}\|_\infty$.
\end{lemma}

\begin{lemma}
\label{lemma-order-Q-beta}
The difference between $Q(\widehat{\bs{\theta}}_{\gamma}^*, \bs{\gamma}^*)$ and $Q(\bs{\theta}^*, \bs{\gamma}^*)$ is
\[
 \|Q(\widehat{\bs{\theta}}_{\bs{\gamma}}^*, \bs{\gamma}^*)-Q(\bs{\theta}^*, \bs{\gamma}^*)\|_\infty  = O_p( ne^{21\rho_n}\log n).
\]
\end{lemma}

Now we are ready to prove  Theorem \ref{Theorem:con}.

\begin{proof}[Proof of Theorem \ref{Theorem:con}]

We construct the Newton iterative sequence to show the consistency. It is sufficient to verify the
Newton-Kantovorich conditions in Lemma \ref{pro:Newton:Kantovorich}.
We set $\bs{\gamma}^*$ as the initial point $\bs{\gamma}^{(0)}$ and $\bs{\gamma}^{(k+1)}=\bs{\gamma}^{(k)} - [Q_c^\prime(\bs{\gamma}^{(k)})]^{-1}Q_c(\bs{\gamma}^{(k)})$.

By Lemma \ref{lemma:alpha:beta-fixed}, $\widehat{\bs{\theta}}_{\gamma^*}$ exists with probability approaching one and satisfies
we have
\[
\| \widehat{\bs{\theta}}_{\gamma^*} - \bs{\theta}^* \|_\infty = O_p\left(  e^{7\rho_{n}} \sqrt{\frac{ \log n}{n} } \right). 
\]
Therefore, $Q_c(\bs{\gamma}^{(0)})$ and $Q_c^\prime(\bs{\gamma}^{(0)})$ are well defined.

Recall the definition of $Q_c(\bs{\gamma})$ and $Q(\bs{\theta}, \bs{\gamma})$ in \eqref{definition-Q} and \eqref{definition-Qc}.
By Lemmas \ref{lemma-diff-F-Q} and \ref{lemma-order-Q-beta}, we have
\begin{eqnarray*}
\|Q_c(\bs{\gamma}^*)\|_\infty  & \le &  \|Q(\bs{\theta}^*, \bs{\gamma}^*)\|_\infty + \|Q(\widehat{\bs{\theta}}_{\bs{\gamma}^*}, \bs{\gamma}^*)-Q(\bs{\theta}^*, \bs{\gamma}^*)\|_\infty\\
&=& O_p\left(  ne^{21\rho_n}\log n \right).
\end{eqnarray*}
By Lemma \ref{lemma-Q-Lip}, $L=n^2e^{19\rho_n} $.
By \eqref{condition-Qc-gamma}, we have
\[
\aleph=\| [Q_c^\prime(\bs{\gamma}^*)]^{-1} \|_\infty = O ( \kappa_n n^{-2}).
\]
Thus,
\[
\delta = \| [Q_c^\prime(\bs{\gamma}^*)]^{-1} Q_c(\bs{\gamma}^*) \|_\infty =  O_p\left(
 \frac{\kappa_n e^{21\rho_n}\log n }{ n }\right  ).
\]
 As a result, if $\kappa_n^2 e^{40\rho_n}=o(n/\log n)$, then
\[
\rho=2\aleph L \delta =
O( \frac{ \kappa_n^2 e^{40\rho_n}  \log n }{n})=o(1).
\]
By Lemma \ref{pro:Newton:Kantovorich}, with probability $1-O(n^{-1})$, the limiting point of the sequence $\{\bs{\gamma}^{(k)}\}_{k=1}^\infty$ exists denoted by $\widehat{\bs{\gamma}}$ and satisfies
\[
\| \widehat{\bs{\gamma}} - \bs{\gamma}^* \|_\infty = O(\delta).
\]
By Lemma \ref{lemma:alpha:beta-fixed}, $\widehat{\bs{\theta}}_{\widehat{\bs{\gamma}}}$ exists and
$( \widehat{\bs{\theta}}_{\widehat{\bs{\gamma}}},\widehat{\bs{\gamma}})$ is the MLE.
It completes the proof.
\end{proof}

\section{Proofs for Theorem \ref{Theorem:binary:central}}
\label{section-theorem2}

Let $\mathbf{h}= (d_1, \ldots, d_n, b_1, \ldots, b_{n-1})^\top$. To show Theorem \ref{Theorem:binary:central}, we need the asymptotic distribution of $S\{\mathbf{h}-\E(\mathbf{h})\}$, which is stated below.

\begin{lemma}\label{lem:binary:central}
If $e^{\rho_n}=o(n^{1/5})$, then for any fixed $k \ge 1$, as
$n\to\infty$, the vector consisting of the first $k$ elements of
$S\{\mathbf{h}-\E(\mathbf{h})\}$ is asymptotically multivariate normal with mean zero
and covariance matrix given by the upper left $k \times k$ block of $S$.
\end{lemma}

Now, we give the proof of Theorem \ref{Theorem:binary:central}.
\begin{proof}[Proof of Theorem \ref{Theorem:binary:central}]

To simplify notations, write  $\lambda_{ij}^\prime = \lambda^\prime(\alpha_i^* + \beta_j^* + Z_{ij}^\top \bs{\gamma}^*)$ and
\[
V= \frac{ \partial F(\bs{\theta}^*, \gamma^*)}{\partial \bs{\theta}^\top}, ~~ V_{\theta\gamma} = \frac{ \partial F(\bs{\theta}^*, \bs{\gamma}^*)}{\partial \bs{\gamma}^\top}.
\]
Recall that $\pi_{ij}=\alpha_i+\beta_j+Z_{ij}^\top \bs{\gamma}$.
By a second order Taylor expansion, we have
\begin{equation}
\label{equ-Taylor-exp}
\lambda( \widehat{\alpha}_i+\widehat{\beta}_j + Z_{ij}^\top \widehat{\bs{\gamma}}) - \lambda(\alpha_i^*+\beta_j^* + \bs{\gamma}^*)
= \lambda_{ij}^\prime (\widehat{\alpha}_i-\alpha_i^*)+\lambda_{ij}^\prime (\widehat{\beta}_j-\beta_j^*) + \lambda_{ij}^\prime Z_{ij}^\top ( \widehat{\bs{\gamma}} - \bs{\gamma})
+ g_{ij},
\end{equation}
where
\[
g_{ij}= \frac{1}{2} \begin{pmatrix}
\widehat{\alpha}_i-\alpha_i^* \\
\widehat{\beta}_j-\beta_j^* \\
\widehat{\bs{\gamma}} - \bs{\gamma}^*
\end{pmatrix}^\top
\begin{pmatrix}
\lambda^{\prime\prime}_{ij}( \tilde{\pi}_{ij} ) & \lambda^{\prime\prime}_{ij}( \tilde{\pi}_{ij} )
& \lambda^{\prime\prime}_{ij}( \tilde{\pi}_{ij} ) Z_{ij}^\top \\
\lambda^{\prime\prime}_{ij}( \tilde{\pi}_{ij} ) & \lambda^{\prime\prime}_{ij}( \tilde{\pi}_{ij} )
& \lambda^{\prime\prime}_{ij}( \tilde{\pi}_{ij} ) Z_{ij}^\top \\
\lambda^{\prime\prime}_{ij}( \tilde{\pi}_{ij} )Z_{ij}^\top
& \lambda^{\prime\prime}_{ij}( \tilde{\pi}_{ij} ) Z_{ij}^\top & \lambda^{\prime\prime}_{ij}( \tilde{\pi}_{ij} ) Z_{ij}Z_{ij}^\top
\end{pmatrix}
\begin{pmatrix}
\widehat{\alpha}_i-\alpha_i^* \\
\widehat{\beta}_j-\beta_j^* \\
\widehat{\bs{\gamma}} - \bs{\gamma}^*
\end{pmatrix},
\]
and $\tilde{\pi}_{ij}$ lies between $\pi_{ij}^*$ and $\widehat{\pi}_{ij}$.
Let $q= \max_{i,j} \| Z_{ij} \|_\infty=O(1)$.
Because $|\lambda^{\prime\prime}(x)|\le e^{\rho_n}$ for all $x\in \R$, we have
\begin{equation}\label{ineq-gij-upper}
\begin{array}{rcl}
|g_{ij}| & \le & \frac{1}{2} \| \widehat{\bs{\theta}} - \bs{\theta}^*\|_\infty^2 + \frac{1}{2}\| \widehat{\bs{\theta}} - \bs{\theta}^*\|_\infty \| \widehat{\bs{\gamma}}-\bs{\gamma}^* \|_{\infty} q + \frac{1}{8} \| \| \widehat{ \bs{\gamma}}-\bs{\gamma}^* \|_{\infty}^2 q^2 \\
& = & O(  \| \widehat{\bs{\theta}} - \bs{\theta}^*\|_\infty^2+  \| \widehat{\bs{\gamma}}-\bs{\gamma}^* \|_{\infty}^2 ).
\end{array}
\end{equation}
Let
\[
g_i=\begin{cases} \sum_{ j\neq i}g_{ij}, & i=1, \ldots, n, \\
\sum_{ j\neq i-n} g_{j, i-n}, & j=n+1, \ldots, 2n-1.
\end{cases}
\]
and  $\mathbf{g}=(g_1, \ldots, g_{2n-1})^\top$.
If $\kappa_n^2e^{43\rho_n}=o(n/\log n)$,
by Theorem 1, we have
\begin{equation}
\label{inequality-gij}
\max_{i=1, \ldots, n(n-1)} |g_i| \le n\max_{i,j} |g_{i,j}| = O_p( e^{15\rho_n} \log n )+ O_p( \frac{\kappa_n^2 e^{43\rho_n} (\log n)^2}{ n } ).
\end{equation}
Let $\mathbf{h}= (d_1, \ldots, d_n, b_1, \ldots, b_{n-1})^\top$.
By writing \eqref{equ-Taylor-exp} into a matrix form, we have
\[
 \E \mathbf{h} - \mathbf{h} = V(\widehat{\bs{\theta}} - \bs{\theta}^*) + V_{\theta\gamma} (\widehat{\bs{\gamma}}-\bs{\gamma}^*) + \mathbf{g},
\]
which is equivalent to
\begin{equation}
\label{expression-beta}
\widehat{\bs{\theta}} - \bs{\theta}^* =-V^{-1}(\mathbf{h}-\E \mathbf{h}) - V^{-1}V_{\theta\gamma} (\widehat{\bs{\gamma}}-\bs{\gamma}^*) - V^{-1} \mathbf{g}.
\end{equation}
We bound the last two remainder terms the above equation as follows.
Let $W=V^{-1} - S$.
Note that  $ne^{-\rho_n} \le u_{i\cdot} \le ne^{\rho_n}$ and
\[
(S g)_i =
\begin{cases}
g_i/u_{i\cdot} + \frac{1}{ u_{\cdot n}} (\sum_{k=1}^n g_k - \sum_{k=n+1}^{2n-1} g_k ) &  i=1, \ldots, n, \\
g_i/u_{\cdot,i-n} + \frac{1}{ u_{\cdot n}} (\sum_{k=1}^n g_k - \sum_{k=n+1}^{2n-1} g_k ) &  i=n+1, \ldots, 2n-1.
\end{cases}
\]
Observe that
\[
\sum_{k=1}^n g_k - \sum_{k=n+1}^{n(n-1)} g_k = \sum_{i=1}^n \sum_{k=1, k\neq i}^n g_{ik} - \sum_{i=1}^n \sum_{k=1,k\neq i}^n g_{ki}
= \sum_{i=1}^n g_{in} .
\]
By \eqref{ineq-gij-upper} and Theorem 1, we have
\[
|\sum_{i=1}^n g_{in} |= O_p( e^{15\rho_n} \log n )+ O_p( \frac{\kappa_n^2 e^{43\rho_n} (\log n)^2}{ n } ).
\]
Therefore,
\begin{equation}\label{eq-Sg-upper}
\| S\mathbf{g} \|_\infty = O_p(\frac{e^{16\rho_n} \log n}{n} )+ O_p( \frac{\kappa_n^2 e^{44\rho_n} (\log n)^2}{ n^2 } ).
\end{equation}
By Proposition 1 in \cite{Yan:Leng:Zhu:2016a} (details in the Supplementary Material), we have
\begin{eqnarray}
\label{eq-Vg-upper}
\| W \mathbf{g} \|_\infty \le \|W\|_\infty \|\mathbf{g} \|_\infty
= O_p(\frac{ e^{20\rho_n} \log n}{n} )+ O_p( \frac{\kappa_n^2 e^{48\rho_n} (\log n)^2}{ n^2 } ).
\end{eqnarray}
By combining \eqref{eq-Sg-upper} and \eqref{eq-Vg-upper}, we have
\begin{equation}\label{eq-V-inverse-g}
\| V^{-1} \mathbf{g} \|_\infty= O_p(\frac{ e^{20\rho_n} \log n}{n} )+ O_p( \frac{\kappa_n^2 e^{48\rho_n} (\log n)^2}{ n^2 } ).
\end{equation}
Further, we have
\begin{equation}
\label{equ-theorem3-333}
\|V^{-1}V_{\theta\gamma} (\widehat{\bs{\gamma}}-\bs{\gamma}^*)\|_\infty
= O_p(  \frac{\kappa_n e^{27\rho_n} \log n}{n}),
\end{equation}
and, with probability at least $1-O(1/n)$,
\begin{equation}\label{equ-theorem3-dd}
\max_{i=1, \ldots, 2n} |[W(\mathbf{h}-\E\mathbf{h})]_i| \lesssim e^{7\rho_n} \frac{ (\log n)^{1/2}}{n}.
\end{equation}
The detailed proof of \eqref{equ-theorem3-333} and \eqref{equ-theorem3-dd} are in the Supplementary Material.

Consequently, by combining \eqref{expression-beta}, \eqref{eq-V-inverse-g}, \eqref{equ-theorem3-333} and \eqref{equ-theorem3-dd},
we have
\[
\widehat{\theta}_i - \theta^*_i = - [S( \mathbf{h} - \E \mathbf{h})]_i -O_p( \frac{\kappa_n e^{23\rho_n} \log n}{n})- O_p(\frac{ e^{20\rho_n} \log n}{n} )- O_p( \frac{\kappa_n^2 e^{48\rho_n} (\log n)^2}{ n^2 } ).
\]
If $\kappa_n e^{23\rho_n} = o( n^{1/2}/\log n)$, then
\[
\widehat{\theta}_i - \theta^*_i = - [S( \mathbf{h} - \E \mathbf{h})]_i +o_p(n^{-1/2}).
\]
Theorem 2 immediately follows from Proposition \ref{lem:binary:central}. It completes the proof.

\end{proof}

\subsection{Proof of Theorem \ref{theorem:covariate:asym}}
The proof strategy is similar to that in \cite{Yan-Jiang-Fienberg-Leng2018}
 for proving
the asymptotic normality of the restricted MLE for ${\bs{\gamma}}$.
We only present the main steps here and details can be found in the Supplementary Material.

Let $T_{ij}$ be an $n$-dimensional column vector with $i$th and $j$th elements ones and other elements zeros. Define
$$s_{ij}(\bs{\theta}, \bs{\gamma}) = (\E a_{ij} - a_{ij}) ( Z_{ij} - V_{\gamma\theta} V^{-1} T_{ij}).$$
Note that $s_{ij}(\bs{\theta}, \bs{\gamma})$, $i,j=1,\ldots,n,i\neq j$, are independent vectors and
\[
\mathrm{Cov} (\sum_{i,j=1,i\neq j }^{2n-1} s_{ij}(\bs{\theta}^*,
\bs{\gamma}^*) ) = \mathrm{Cov}( Q(\bs{\theta}^*,
\bs{\gamma}^*) ) - V_{\theta\gamma}^\top V^{-1} F(\bs{\theta}^*,
\bs{\gamma}^*) ) )=H(\bs{\theta}^*, \bs{\gamma}^*).
\]
By Lyapunov's  central limit theorem (\cite{Billingsley:1995}, p 362), we have the following Lemma.
\begin{lemma}\label{lemma-chcs}
For any nonzero fixed vector $c=(c_1, \ldots c_p)^\top$, if $e^{\rho_{n}}=o(n^{2/5})$, then \\
 $(c^\top H(\bs{\theta}^*,\bs{\gamma}^*) c )^{-1/2}c^\top\sum_{i,j=1,i\neq j}^{2n-1}
s_{ij} (\bs{\theta}^*,\bs{\gamma}^*)$ converges in distribution to the standard normal distribution.
\end{lemma}

Now, we give the proof of Theorem 3.
\begin{proof}[Proof of Theorem 3]
Assume that the condition in Theorem 1 hold.
A mean value expansion gives
\[
 Q_c( \widehat{\bs{\gamma}} ) - Q_c(\bs{\gamma}^*) =  \frac{\partial Q_c(\bar{\bs{\gamma}}) }{ \partial \bs{\gamma}^\top }  (\widehat{\bs{\gamma}}-\bs{\gamma}^*),
\]
where $\bar{\bs{\gamma}}$ lies between $\bs{\gamma}^*$ and $\widehat{\bs{\gamma}}$.
By noting that $Q_c( \widehat{\bs{\gamma}} )=0$, we have
\[
\sqrt{N}(\widehat{\bs{\gamma}} - \bs{\gamma}^*) = -
\Big[ \frac{1}{N}  \frac{\partial Q_c(\bar{\bs{\gamma}}) }{ \partial \bs{\gamma}^\top } \Big]^{-1}
\times \frac{1}{\sqrt{N}}　Q_c(\bs{\gamma}^*). 
\]
Note that the dimension of $\bs{\gamma}$ is fixed. By Theorem 1, we have
\[
\frac{1}{N}  \frac{\partial Q_c(\bar{\bs{\gamma}}) }{ \partial \bs{\gamma}^\top }
\stackrel{p}{\to } \bar{H}=\lim_{N\to\infty} \frac{1}{N}H(\bs{\theta}^*, \bs{\gamma}^*).
\]
Write $\widehat{\bs{\theta}}^*$ as $\widehat{\bs{\theta}}_{\gamma^*}$ for convenience.
 Therefore,
\begin{equation}\label{eq:theorem4:aa}
\sqrt{N} (\widehat{\bs{\gamma}} - \bs{\gamma}^*) = - \bar{H}^{-1} \cdot \frac{1}{\sqrt{N}} {Q}( \widehat{\bs{\theta}}^*, \bs{\gamma}^*)  + o_p(1).
\end{equation}
By applying a third order Taylor expansion to $Q( \widehat{\bs{\theta}}^*, \bs{\gamma}^*)$, it yields
\begin{equation*}\label{eq:gamma:asym:key}
\frac{1}{\sqrt{N}}  Q(\widehat{\bs{\theta}}^*, \bs{\gamma}^*) = S_1 + S_2 + S_3,
\end{equation*}
where
\begin{equation*}
\begin{array}{l}
S_1  =  \frac{1}{\sqrt{N}}  Q(\bs{\theta}^*, \bs{\gamma}^* )
+ \frac{1}{\sqrt{N}}
\Big[\frac{\partial  Q(\bs{\theta}^*, \bs{\gamma}^* ) }{\partial \bs{\theta}^\top } \Big]( \widehat{\bs{\theta}}^* - \bs{\theta}^* ), \\
S_2  =   \frac{1}{2\sqrt{N}} \sum_{i=1}^{2n-1} \Big[( \widehat{\theta}_i^* - \theta_i^* )
\frac{\partial^2 Q(\bs{\theta}^*, \bs{\gamma}^* ) }{ \partial \theta_i \partial \bs{\theta}^\top } (\bs{\theta}^*, \bs{\gamma}^* )
\times ( \widehat{\bs{\theta}}^* - \bs{\theta}^* ) \Big],  \\
S_3  =  \frac{1}{6\sqrt{N}} \sum_{i=1}^{2n-1} \sum_{j=1}^{2n-1} \{ (\widehat{\theta}_i^* - \theta_i^*)(\widehat{\theta}_j^* - \theta_j^*)
\Big[   \frac{ \partial^3 Q(\bar{\bs{\theta}}^*, \bs{\gamma}^*)}{ \partial \theta_i \partial \theta_j \partial \bs{\theta}^\top } \Big]
(\widehat{\bs{\theta}}^*  - \bs{\theta}^* )\},
\end{array}
\end{equation*}
and $\bar{\bs{\beta}}^*=t\bs{\beta}^*+(1-t)\widehat{\bs{\beta}}^*$ for some $t\in(0,1)$.
We will show that (1) $S_1$ asymptotically follows a multivariate normal distribution;
(2) $S_2$ is a bias term; (3) $S_3$ is an asymptotically negligible remainder term.
Specifically, they are accurately characterized as follows:
\begin{align*}
&S_1  =  \frac{1}{\sqrt{N}} \sum_{i,j=1,i\neq j}^{2n-1} s_{ij}(\bs{\theta}^*, \bs{\gamma}^*) + O_p( \frac{ e^{21\rho_n}\log n}{n}),\\
&S_2  =  \frac{1}{2\sqrt{N}}\left(\sum_{i=1}^{n} \frac{ \sum_{j=1,j\neq i}^{n} \lambda^{\prime\prime}(\pi_{ij}^*) Z_{ij} }{ \sum_{j=1,j\neq i}^{n} \lambda^{\prime}(\pi_{ij}^*)}+ \sum_{j=1}^{n} \frac{ \sum_{i=1,i\neq j}^{n} \lambda^{\prime\prime}(\pi_{ij}^*) Z_{ij} }{ \sum_{i=1,i\neq j}^{n} \lambda^{\prime}(\pi_{ij}^*) }\right) + O_p(\frac{  e^{9\rho_n} (\log n)^{2}}{ n^{1/2} } ),\\
&\|S_3\|_\infty =  O_p( \frac{ (\log n)^{3/2}e^{22\rho_n} }{n^{1/2}} ).
\end{align*}
We defer the proofs of the above equations to supplementary material.
Substituting the above equations into \eqref{eq:theorem4:aa} then gives
\[
\sqrt{N}(\widehat{\bs{\gamma}}- \bs{\gamma}^*) = - \bar{H}^{-1} B_* + \bar{H}^{-1} \times \frac{1}{\sqrt{N}}  \sum_{i< j}
s_{ij} (\bs{\beta}^*, \bs{\gamma}^*) + O_p\left( \frac{  e^{9\rho_n} (\log n)^{2}}{ n^{1/2} }  \right).
\]
If $e^{9\rho_n} =o(n^{1/2}/(\log n)^2)$,
then Theorem 3 immediately follows from Proposition \ref{lemma-chcs}.
\end{proof}

\setlength{\itemsep}{-1.5pt}
\setlength{\bibsep}{0ex}
\bibliography{reference3}
\bibliographystyle{apa}

\end{document}